\documentclass{amsart}
\usepackage{tikz}
\usepackage{amssymb}
\usepackage{amsmath}
\usepackage{bbm, dsfont}
\usepackage{graphics}
\usepackage{mathrsfs}
\usepackage{color}
\DeclareMathOperator{\Pic}{Pic}
\DeclareMathOperator{\lcm}{lcm}

\title[BHCR-Mirror Symmetry for K3 surfaces]{The Berglund-H\"ubsch-Chiodo-Ruan mirror symmetry for  K3 surfaces}

\author{Michela Artebani}

\address{Departamento de Matem\'atica, Universidad de Concepci\'on, Casilla 160-C, Concepci\'on, Chile}
			
\email{martebani@udec.cl}

\author{Samuel Boissi\`ere}

\address{Laboratoire de Math\'ematiques et Applications, UMR CNRS 6086,
			Universit\'e de Poitiers, T\'el\'eport 2, Boulevard Marie et Pierre Curie,
			F-86962 FUTUROSCOPE CHASSENEUIL}

\email{Samuel.Boissiere@math.univ-poitiers.fr}

\urladdr{http://www-math.sp2mi.univ-poitiers.fr/$\sim$sboissie/}

\author{Alessandra Sarti}

\address{Laboratoire de Math\'ematiques et Applications, UMR CNRS 6086,
			Universit\'e de Poitiers, T\'el\'eport 2, Boulevard Marie et Pierre Curie,
			F-86962 FUTUROSCOPE CHASSENEUIL}
			
\email{sarti@math.univ-poitiers.fr}

\urladdr{http://www-math.sp2mi.univ-poitiers.fr/$\sim$sarti/}

\usepackage[all]{xy}
\usepackage{enumerate}

\newtheorem{lemma}{Lemma}
\newtheorem{pro}{Proposition}
\newtheorem{cor}{Corollary}

\newtheorem{theorem}{Theorem}[section]
\theoremstyle{definition}
\newtheorem{definition}[theorem]{Definition}
\newtheorem{example}[theorem]{Example}
\newtheorem{remark}[theorem]{Remark}

\DeclareMathOperator{\SL}{SL}
\DeclareMathOperator{\Cl}{Cl}

\DeclareMathOperator{\rank}{rk}
\DeclareMathOperator{\diag}{diag}

\DeclareMathOperator{\Fix}{Fix}

\DeclareMathOperator{\Ext}{Ext}

\DeclareMathOperator{\rk}{rk}
\DeclareMathOperator{\Aut}{Aut}

\newcommand{\IC}{\mathbb{C}}
\newcommand{\IQ}{\mathbb{Q}}
\newcommand{\IZ}{\mathbb{Z}}

\newcommand{\IR}{\mathbb{R}}
\newcommand{\IP}{\mathbb{P}}
\newcommand{\KK}{\mathbf{K}}
\newcommand{\calO}{\mathcal{O}}


\thanks{The first author has been 
supported by  Proyecto FONDECYT Regular N. 1110249}
\begin{document}

\begin{abstract}
We prove  that the mirror symmetry of  Berglund-H\"ubsch-Chiodo-Ruan, applied to K3 surfaces with a non-symplectic involution, coincides with the mirror symmetry described by Dolgachev and Voisin.
\end{abstract}

\maketitle

\section{Introduction}
Berglund and H\"ubsch in~\cite{BH} described a very concrete construction of mirror pairs of Calabi-Yau manifolds given as hypersurfaces in some weighted projective spaces.  Later, Chiodo and Ruan in~\cite{ChiodoRuan} proved that the transposition rule of Berglund-H\"ubsch provides pairs of Calabi-Yau manifolds whose Hodge diamonds have the symmetry required in mirror symmetry. 
In this paper we apply the transposition rule to certain K3 surfaces carrying a non-symplectic involution and we relate this to a mirror construction between families of lattice polarized K3 surfaces due to Dolgachev and Nikulin \cite{dn, nikulinmir, dolgachevmirror}, Voisin \cite{voisin} and Borcea \cite{Borcea}. Since in particular the results
of \cite{dolgachevmirror} and \cite[Lemma 2.5 and \S 2.6]{voisin} were fundamental for our Theorem \ref{main} (see subsection \ref{DVmirror}) we will refer to such families as ``Dolgachev-Voisin mirror families''.
Our main theorem is that the transposition rule by Berglund and H\"ubsch, in this case, provides pairs of K3 surfaces which belong to 
the Dolgachev-Voisin mirror families.

Let $W$ denote a Delsarte type polynomial, i.e. a polynomial having as many monomials as variables  (this will be called ``potential'' in the sequel, following the terminology of physicists). Assume that the matrix of exponents of $W$ is invertible, that $\{W=0\}$ has an isolated singularity at the origin and that it defines a well-formed hypersurface in some normalized weighted projective space. We denote by $\Aut(W)$ the group of diagonal symmetries of $W$, by $\SL(W)$ the group of diagonal symmetries of determinant one, and by $J_W$ the monodromy group of the affine Milnor fibre associated to $W$. For any subgroup $G\subset\Aut(W)$, we denote by $G^T$ the ``transposed'' group of automorphisms of the ``transposed'' potential $W^T$ (see section~\ref{s:BHCR} for their definition). The main result of this paper is the following.

\begin{theorem}\label{main}
Let $W$ be a K3 surface defined by a non-degenerate and invertible potential of the form:
$$
x^2=f(y,z,w).
$$
 in some weighted projective space.
Let $G_W$ be a subgroup of  $\Aut(W)$ such that $J_W\subset G_W\subset \SL(W)$. Put $\widetilde{G_W}:=G_W/J_W$ and $\widetilde{G_W^T}:=G_W^T/J_{W^T}$. Then the Berglund-H\"ubsch-Chiodo-Ruan mirror orbifolds $[W/\widetilde{G_W}]$ and $[W^T/\widetilde{G_W^T}]$ belong to the mirror families of
Dolgachev and Voisin.
\end{theorem} 

The Berglund-H\"ubsch-Chiodo-Ruan (BHCR for short) mirror symmetry applies to Calabi-Yau varieties in weighted projective spaces which are not necessarily Gorenstein. As remarked by Chiodo and Ruan in \cite[Section 1]{ChiodoRuan}, this is the main difference with Batyrev mirror symmetry~\cite{batyrev}.
Most of our  K3 surfaces are not contained in a Gorenstein weighted projective space.

The paper is organized as follows.
In section 2 we give some preliminaries about hypersurfaces in weighted projective spaces,
potentials and the Berglund-H\"ubsch construction.
In section 3 we describe the group $\Aut(W)$ of diagonal automorphisms of a potential 
and we define the transposed group $G^T$ of a subgroup $G$ of $\Aut(W)$.
Section 4 contains preliminary facts about non-symplectic involutions on K3 surfaces and 
introduces the Dolgachev-Voisin mirror construction.
Section 5 deals with K3 surfaces defined by a potential as in the statement of Theorem \ref{main}:
we study their singularities and we determine the basic invariants of the non-symplectic involution $x\mapsto -x$.
In section 6 we give the proof of Theorem \ref{main}.\\

{\em Aknowledgements.} We thank Alessandro Chiodo and Antonio Laface for many helpful discussions.

\section{The Berglund-H\"ubsch-Chiodo-Ruan construction}
\label{s:BHCR}
\subsection{Hypersurfaces in weighted projective spaces}\label{weighted} 
We start recalling some basic facts about hypersurfaces in weighted projective spaces, see for example \cite{fletcher}.
Let $x_1,\ldots,x_n$ be affine coordinates on $\IC^{n}$, $n\geq 3$, and let $(w_1,\ldots,w_n)$ be a sequence of positive weights. The group $\IC^*$ acts on $\IC^n$ by  
$$
\lambda (x_1,\ldots,x_n)=(\lambda^{w_1}x_1,\ldots,\lambda^{w_n}x_n)
$$ 
 and the {\it weighted projective space} $\IP(w_1,\ldots,w_n)$ is the quotient $(\IC^{n}\backslash\{0\})/\IC^*$. The weighted projective space is called {\it normalized} if 
$$
\gcd(w_1,\ldots,\widehat{w_i},\ldots,w_n)=1\,\,\mbox{for }\,i=1,\dots,n.
$$  
Weighted projective spaces are singular in general and the singularities arise only on the fundamental simplex $\Delta$ with vertices in the points $P_i:=(0,\ldots,0,1,0,\ldots,0)$, $i=1,\dots,n$. The vertices are singularities of type $1/w_i(w_1,\ldots,\widehat{w_i},\ldots,w_n)$ and they are not necessarily isolated, since the higher dimensional toric strata of $\Delta$ can be singular too. For example, if $h_{i,j}:=\gcd(w_i,w_j)>1$,  then the generic point of the edge $P_iP_j$ is a singularity of type $1/h_{i,j}(w_1,\ldots,\widehat{w_i},\ldots,\widehat{w_j},\ldots, w_n)$. The weighted projective space $\IP(w_1,\ldots,w_n)$ has Gorenstein singularities if and only if $w_j|\sum_{i=1}^{n}w_i$ for all $j$. This is also equivalent to say that the weighted projective space is Fano or finally, regarding $\IP(w_1,\ldots,w_n)$ as toric variety, that its associated polytope is reflexive \cite[Section 3.5]{coxkatz}. 

 A quasihomogeneous polynomial $W(x_1,\ldots,x_n)$ of total degree $d$ defines a hypersurface in 
$\IP(w_1,\ldots,w_n)$, which is also denoted by $W$ in the sequel.
\begin{definition}
The hypersurface $W$ is called
\begin{itemize} 
\item {\it well-formed}  if $\IP(w_1,\ldots,w_n)$ is normalized and 
$$\gcd(w_1,\ldots,\widehat{w_i},\ldots,\widehat{w_j},\ldots, w_n) \mbox{ divides } d$$
for all $i,j=1,\dots,n$;
\item  {\it quasismooth} if it is well-formed and the polynomial $W$ is non-degenerate, i.e. its affine cone is smooth outside its vertex $(0,\ldots,0)$;
\item  {\it Calabi-Yau} ({\it K3 surface} in the two-dimensional case) if it 
has canonical singularities (in particular $W$ is Gorenstein), its canonical bundle is trivial and $H^i(W,\calO_W)=0$ for all $i=1,\ldots,n-3$.
\end{itemize} 
\end{definition}
Observe that by \cite[Lemma 1.12]{cortigo} a well-formed and quasismooth hypersurface $W$ in $\IP(w_1,\ldots,w_n)$ is Calabi-Yau if and only if $d=\sum_{i=1}^{n}w_i$. 
Reid in \cite{reid} and Yonemura in  \cite{yonemura}  give a list of all possible
families of K3 surfaces in  weighted projective spaces. These are 95 in total and only 14 of the weighted projective spaces are Gorenstein. For each type Reid describes the singularities of the K3 surface. By \cite{cortigo} the 95 projective spaces have canonical singularities, and in fact one can determine 104 families of weights such that the weighted projective spaces have canonical singularities. However in 9 cases
one can not obtain K3 surfaces with canonical singularities \cite[Theorem 1.17]{cortigo}.

Finally, we recall that the genus of a smooth curve $C_d$ of total degree $d$ in $\IP(w_1,w_2,w_3)$ is given by the formula:
\begin{eqnarray}\label{genusweighted}
g(C_d)=\frac{1}{2}\left(\frac{d^2}{w_1w_2w_3}-d\sum_{i>j}\frac{\gcd(w_i,w_j)}{w_iw_j}+\sum_{i=1}^{3}\frac{\gcd(d,w_i)}{w_i}-1 \right).
\end{eqnarray}


\subsection{Invertible potentials}\label{general}
We briefly recall the mirror construction of Berglund-H\"ubsch in \cite{BH} and  Chiodo-Ruan in \cite{ChiodoRuan}.
Consider a {\it potential}:
$$
W(x_1,\ldots,x_n)=\sum_{i=1}^{n}\prod_{j=1}^{n}x_j^{a_{ij}},
$$
that is a polynomial in $n$ variables containing $n$ monomials. Since we have $n$ monomials 
it is not a restriction to consider all the coefficients to be equal to $1$, so that 
a potential is identified by the matrix 
$
A:=(a_{ij})_{i,j=1,\ldots,n}
.$
The potential is called {\it invertible} if the matrix $A$ is invertible over $\IQ$. In this case 
we denote by
$
A^{-1}:=(a^{ij})_{i,j=1,\ldots,n}
$
the inverse matrix and define the {\it charge} $q_i:=\sum_{j=1}^{n} a^{ij}$ as the sum of the entries of the $i$-th row of $A^{-1}$.
Clearly the charges $q_i$ satisfy:
$$
A \left(\begin{array}{c} q_1\\\vdots\\q_n\end{array}\right)=\left(\begin{array}{c} 1\\\vdots\\1\end{array}\right).
$$
Let $d$ be the least common denominator of the charges and let $w_i:=dq_i$. Then $\{W=0\}$ defines a hypersurface $W$ in $\IP(w_1,\ldots,w_n)$ of total degree $d$, which we assume to be well-formed and quasismooth. 
Observe that, by \cite[Lemma 1.12]{cortigo}, $W$ is Calabi-Yau if and only if $\sum_iq_i=1$. 
By \cite[Proposition 6]{dimca}, if the weighted projective space is normalized and the hypersurface is quasismooth of dimension $\geq 3$, then it is well-formed. In the case of K3 surfaces this is also true and can be verified by checking in Reid's list. 

By \cite[Theorem 1]{kreuzerskarke} a potential $W$  is invertible and non-degenerate (i.e. the corresponding hypersurface is quasismooth) if and only if it can be written as a sum of invertible potentials of {\it atomic types}:
$$
\begin{array}{lll}
W_{fermat}&:=&x^a,\\
W_{loop}&:=&x_1^{a_1}x_2+x_2^{a_2}x_3+\ldots +x_{n-1}^{a_{n-1}}x_n+x_n^{a_n}x_1,\\
W_{chain}&:=&x_1^{a_1}x_2+x_2^{a_2}x_3+\ldots +x_{n-1}^{a_{n-1}}x_n+x_n^{a_n}.
\end{array}
$$ 
If $W$ is a fermat type polynomial (i.e. sum of $W_{fermat}$) then $\{W=0\}$ defines a hypersurface in a Gorenstein weighted projective space. In the other cases this is not true in general.

\subsection{The Berglund-H\"ubsch-Chiodo-Ruan construction}


Given an invertible and non-degenerate potential $W$ as in the previous subsection, we consider the group of diagonal automorphisms:
$$
\Aut(W):=\{\gamma=(\gamma_1,\ldots,\gamma_n)\in(\IC^*)^n\,|\,W(\gamma_1x_1,\ldots,\gamma_nx_n)=W(x_1,\ldots,x_n)\}
$$
and its subgroup
$$
\SL(W):=\Aut(W)\cap \SL_{n}(\IC).
$$
To each column of $A^{-1}$  we associate the diagonal matrix 
$$
\rho_j:=\diag(\exp(2\pi i a^{1,j}),\ldots,\exp(2\pi i a^{n,j}))\in\Aut(W)
$$
and we define the matrix $j_W$ to be the product 
$$
\rho_1\cdots \rho_n=\diag(\exp(2\pi i q_1),\ldots,\exp(2\pi i q_n)).
$$ 
Observe that  the group $J_W$ generated by $j_W$ is cyclic of order $d$ and 
acts trivially on the hypersurface $W$, 
since it acts trivially on the weighted projective space $\IP(w_1,\dots,w_n)$. 
In what follows we assume the hypersurface $W$ to be Calabi-Yau. Then $\sum_{i} q_i=1$, so that  $J_W\subset\SL(W)$.

Let $G_W$ be a group of diagonal automorphisms such that $J_W\subset G_W\subset \SL(W)$ 
and  let $\widetilde{G_W}:=G_W/J_W$.
We will now construct a potential $W^T$ and a group $G_{W}^T$. The potential $W^T$ is defined by transposing the matrix $A$:
$$
W^T:=W^T(x_1,\ldots,x_n)=\sum_{i=1}^{n}\prod_{j=1}^{n}x_j^{a_{ji}}.
$$ 
Similarly as before, we denote by $q_j^T$ the charge of $W^T$, which is the sum of the entries of the $j$-th column of $A^{-1}$.  
Observe that $\sum_{j} q_j=\sum_{j} q_j^T=\sum_{i,j}a^{i,j}=1$.
Since the potential $W^T$ is non-degenerate by the classification in \cite{kreuzerskarke} and the charges satisfy $\sum_{j} q_j^T=1$, then  it is easy to show that the equation $\{W^T=0\}$ defines a variety in a normalized weighted projective space.   
By \cite[Proposition 6]{dimca}, if $n\geq 5$ the hypersurface is well-formed, so that the potential $W^T$ defines a Calabi-Yau variety.  
This is true also if $n=3,4$, as can be checked by a quick case-by-case analysis.  
\begin{remark}
Without the condition $\sum_{i}q_i=1$ it is not true that the equation $\{W^T=0\}$
 defines a variety in a normalized projective space. For example
 ${W=x_1^5x_2+x_2^2x_3+x_3^3x_4+x_4^9}$ defines a surface in $\IP(7,19,16,6)$ and $W^T$ defines a surface in $\IP(9,18,9,4)$, which is clearly not normalized.
\end{remark}
The group $G_W^T$ is defined by Krawitz in \cite{krawitz} as:
$$
G_W^T=\left\{\prod_{j=1}^{n}(\rho_j^T)^{m_j}\,|\,\prod_{j=1}^n x_j^{m_j} \,\mbox{is}\, G_W\mbox{-invariant}\right\},
$$
where the definition of the automorphisms $\rho_j^T$ of $W^T$ is similar to the definition of $\rho_j$ using  the matrix 
$A^T$. Equivalent definitions for the group $G_W^T$ will be given in the next section.
The group satisfies $J_{W^T}\subset G_W^T\subset \SL_{W^T}$. Putting $\widetilde{G_W^T}:=G_W^T/J_{W^T}$, we have the following result.
\begin{theorem}\cite[Theorem 2]{ChiodoRuan}
The Calabi-Yau orbifolds $[W/\widetilde{G_W}]$ and $[{W^T}/\widetilde{G_W^T}]$ form a mirror pair,
i.e. we have
$$
H_{\rm{CR}}^{p,q}([W/\widetilde{G_W}],\IC)\cong H_{\rm{CR}}^{n-2-p,q}([{W^T}/\widetilde{G_W^T}],\IC)
$$
where $H_{\rm{CR}}(-,\IC)$ stands for the Chen-Ruan orbifold cohomology.
\end{theorem}

The previous result clearly gives no information in the case of K3 surfaces, since all K3 surfaces have the same Hodge diamond.
However, it is a strong motivation for considering this as a good mirror correspondence even in the two-dimensional case.

We now show that the action of  $\SL(W)$ is symplectic.
\begin{pro}\label{actions}
Let $W$ be a non-degenerate potential defining a Calabi-Yau manifold 
in $\IP(w_1,\ldots,w_n)$. Then the action
of $\SL(W)$ on the volume form is trivial.  
\end{pro}
\proof
We can write the volume form locally for $x_1\not=0$ and 
$\frac{\partial W}{\partial x_n}\not= 0$ as
$$
\xi:=\frac{dx_2\wedge\ldots\wedge dx_{n-1}}{\frac{\partial W}{\partial x_n}}.
$$
Let $g=(\exp(2\pi i\alpha_1),\ldots,\exp(2\pi i\alpha_n))\in \SL(W)$. We can normalize $g$ 
multiplying by $\exp(2\pi i (-\alpha_1/w_1))$, so that we obtain
$g=(1,\exp(2\pi i \beta_2),\ldots, \exp(2\pi i \beta_{n}))$ with $\beta_i=\alpha_i-(w_i/w_1)\alpha_1$. If we apply this transformation to $W$, this is    multiplied by $\exp(2d\pi i (-\alpha_1/w_1))$.
We have that 
$$
g\frac{\partial W}{\partial x_n}=\exp(2\pi i(-\beta_n-(\alpha_1/w_1)d))W.
$$
Hence the form $\xi$ is multiplied by $\exp(2\pi i\delta)$, with
$$
\delta=\beta_1+\ldots+\beta_{n}+\frac{\alpha_1}{w_1}d=\alpha_1+\alpha_2+\ldots+\alpha_n\in \IZ.
\qed$$

\section{The group of diagonal automorphisms}
\subsection{Description of $\Aut(W)$}
Let $W:\IC^n\to \IC$ be a non-degenerate, invertible potential
and let $A=(a_{ij})_{i,j}\in {\rm GL}(n,\IQ)$ be the associated matrix. 
In this section we will describe the group $\Aut(W)$
of diagonal automorphisms of $W$ and its subgroups.
We start observing the following: 
\begin{lemma} $\Aut(W)$ is finite. 
\end{lemma}
\begin{proof} Since $\Aut(W)$ is abelian, it is enough to prove that 
any of its elements has finite order. 
If $\gamma\in \Aut(W)$ then $\prod_{j=1}^n \gamma_j^{a_{ij}}=1$ for any $i\in \{1,\dots,n\}$,
in particular $\prod_{j=1}^n |\gamma_j|^{a_{ij}}=1$.
Thus, taking the logarithm we obtain that 
$$(\ln|\gamma_1|,\dots,\ln|\gamma_n|)\in \ker(A)=\{0\},$$
which implies that $|\gamma_i|=1$.
Thus $\gamma_i=\exp(2\pi i a_i)$, $a_i\in \IR$, and the 
previous condition on $\gamma$ can be translated as $A\cdot a\in \IZ^n$, 
where $a=(a_1,\dots,a_n)$.
Since $A$ has integral entries, then $a\in \IQ^n$, so that $\gamma$ has finite order.
\end{proof}
After writing $\gamma=(\exp(2\pi i a_1),\dots,\exp(2\pi i a_n))$ with $a_i\in \IQ$, we can identify $\Aut(W)$ with 
$$\{a=(a_1,\dots,a_n)\in (\IQ/\IZ)^n|A\cdot a\in \IZ^n\}=A^{-1}\IZ^n/\IZ^n.$$
In particular, $|\Aut(W)|=|\det(A)|$.
Similarly, we will identify $\Aut(W^T)$ with $(A^T)^{-1}\IZ^n/\IZ^n$.
We observe that, since $A$ and $A^T$ have the same Smith normal form, then $\Aut(W)\cong \Aut(W^T)$. 

By means of the previous description,  since $\Aut(W)$ is generated by the columns of $A^{-1}$, 
we obtain the following result (see also \cite[Lemma 1.6]{krawitz}).
\begin{pro}\label{order}\text{}
\begin{itemize}
\item[1)] $\Aut(W_{fermat})\cong\IZ/a\IZ$ and a generator is $\frac{1}{a}$,
\item[2)] $\Aut(W_{loop})\cong\IZ/(a_1\cdots  a_n+(-1)^{n+1})\IZ$ 
and a generator is $(\varphi_1,\dots,\varphi_n)$,
where
$$\varphi_1:=\frac{(-1)^n}{\Gamma},\quad \varphi_i:=\frac{(-1)^{n+1-i}a_1\cdots a_{i-1}}{\Gamma},\ i\geq 2.$$
\item[3)] $\Aut(W_{chain})\cong \IZ/(a_1\cdots a_n)\IZ$ and a generator 
is given by $(\varphi_1,\dots,\varphi_n)$, where
$$\varphi_i:=\frac{(-1)^{n+i}}{a_i\cdots a_n}.$$
\end{itemize}
\end{pro}

 A subgroup $G$ of $\Aut(W)$ is given by $C^{-1}\IZ^n/\IZ^n$, where $C\in {\rm M}(n,\IZ)$ is a matrix invertible over $\IQ$  such that the columns of $C^{-1}\in {\rm M}(n,\IQ)$ are spanned by the columns of $A^{-1}$, i.e. 
$C^{-1}=A^{-1}B$ for some $B\in {\rm M}(n,\IZ)$.

\begin{remark}\label{J} Let $J_W=\langle q\rangle$, where $q=(q_1,\dots,q_n)$ is the vector of charges 
and let $C_0\in {\rm M}(n,\IZ)$ be such that $J_W=C_0^{-1}\IZ^n/\IZ^n$.
In this case the columns of $C_0^{-1}$ are a basis of the lattice $L$ 
generated by the canonical basis $e_1,\dots,e_n$ and the vector $q$.
Such a basis can be obtained as follows: let $w=(w_1,\dots,w_n)$ be the vector of weights and let 
$M\in {\rm GL}(n,\IZ)$ such that $Mw=e_1$ (this is possible since $w$ is primitive), then
a basis of $L$ is given by $q,M^{-1}e_2,\dots,M^{-1}e_n$:
$$
C_0^{-1}=\left(\begin{matrix}
q & M^{-1}e_2& \dots & M^{-1}e_n
\end{matrix}\right)
$$
 \end{remark}
In what follows  $e$ will be the column vector with all entries equal to $1$.
\begin{lemma}\label{js} 
$J_W\subset  G$ if and only if $B^{-1}e\in \IZ^n$ and $G\subset {\rm SL}(W)$ if and only if $(C^T)^{-1}e\in \IZ^n.$
\end{lemma}
\begin{proof}
Recall that $J_W$ is generated by $q=A^{-1}e$. Thus $J_W\subset G$ if and only if $CA^{-1}e=B^{-1}e\in \IZ^n$. The group $G$ is contained in ${\rm SL}(W)$ if and only if $\sum_{i=1}^n a_i=a\cdot e\in \IZ$ for all $a=(a_1,\dots,a_n)\in G$. Equivalently
$$(C^{-1}u)^T e=u^T(C^T)^{-1} e\in \IZ$$
for all $u\in \IZ^n$, i.e. $(C^T)^{-1} e\in \IZ^n$.
\end{proof}

\subsection{Description of $G_W^T$}
Given a subgroup $G=C^{-1}\IZ^n/\IZ^n$ of $\Aut(W)$, where $C=A^{-1}B\in {\rm M}(n,\IZ)$,
we define the {\em transpose group} $G^T$ in $\Aut(W^T)$ as:
$$G^T:=(B^T)^{-1}\IZ^n/\IZ^n.$$
 
As a consequence of the previous description of the group $G$ we have the following properties.

\begin{pro}\label{grouprop}  
\begin{itemize}
\item[]
\item[1)]  $|G|=|\det(C)|$ and $|G^T|=|\det(B)|$,
\item[2)] $(G^T)^T=G$,
\item[3)] $\{0\}^T=\Aut(W^T)$ and $\Aut(W)^T=\{0\}$,
\item[4)] $J_W^T={\rm SL}(W^T)$,
\item[5)] if $G_1\subset G_2$, then $G_2^T\subset G_1^T$ and $G_2/G_1\cong G_1^T/G_2^T$.
\end{itemize}
 \end{pro}
 \begin{proof} We will prove statement 4), the remaining ones follow easily from the definition.
Let $C_0^{-1}$ be a matrix corresponding to $J_W$ as in Remark \ref{J} and let $B_0=AC_0^{-1}$.
By Lemma \ref{js}, $J_W^T$ is contained in $\SL(W^T)$ if $B_0^{-1}e\in \IZ^n$.
Equivalently:
$$C_0A^{-1}e=C_0q\in \IZ^n,$$
which clearly holds since $q=C_0^{-1}e_1$.

Conversely, let $a=(A^T)^{-1}v=(B_0^T)^{-1}(C_0^T)^{-1}v\in \Aut(W^T)$, $v\in \IZ^n$, such that $a\cdot e\in \IZ$.
We show that $(C_0^T)^{-1}v\in \IZ^n$.
The condition $a\cdot e\in \IZ$ is equivalent to:
$$a^TAq=a^TAC_0^{-1}e_1=v^TC_0^{-1}e_1\in \IZ,$$
i.e. $(C_0^T)^{-1}v\cdot e_1\in \IZ^n$. 
Since the columns of $C_0^{-1}$, except for the first one, have integral entries, 
this is enough to prove that $(C_0^T)^{-1}v\in \IZ^n$. Thus $a\in J_W^T$.
\end{proof}

\begin{remark}
By  Proposition \ref{grouprop} it follows that  
$J_W\subset G$ if and only if $G^T\subset\SL(W^T)$ and 
$J_{W^T}\subset G^T$ if and only if $G\subset \SL(W)$.
Moreover, $\SL(W^T)/J_{W^T}$ is isomorphic to
 $\SL(W)/J_W$,
so that $\SL(W^T)=J_{W^T}$ if and only if $J_W=\SL(W)$.
\end{remark}

 \subsection{The group $\SL(W)$}
 We will now determine the order of the subgroup 
 $\SL(W)=\Aut(W)\cap \SL_n(\IC)$.
  \begin{cor}\label{sl}
 The order of $\SL(W)$ is equal to $|\det(A)|/d^T$, where $d^T$ is the least common denominator of the charges of $W^T$.
 \end{cor}
 \begin{proof}
 By Proposition \ref{grouprop}, $\SL(W^T)=J_{W}^T$ and
 $|\SL(W^T)|=|\det(B_0)|$ where $B_0=AC_0^{-1}$ and $C_0$ is given in Remark \ref{J}.
 Observe that  
 $$MC_0^{-1}= \left(\begin{matrix}
e_1/d & e_2& \dots & e_n
\end{matrix}\right)$$
since $Mq=M(w/d)=e_1/d$. 
Thus 
$$|\SL(W^T)|=|\det(A)\det(C_0^{-1})|=|\det(A)\det(MC_0^{-1})|=\frac{|\det(A)|}{d}.$$
Changing $W$ with $W^T$ we get the statement.
 \end{proof}

\begin{pro}\label{sl4} Let $W:\IC^{4}\to \IC$ be a well-formed potential 
of the form
$$W(x,y,z,w)=x^2-f(y,z,w)$$
and let $A=(a_{ij})_{i,j=1,2,3}$ be the matrix associated to $f$.
\begin{itemize}\itemsep 0.25cm
\item If $f$ is of chain type, then 
$|\SL(W)|=2\gcd(a_1a_2a_3,1-a_1+a_1a_2).$
\item If $f$ is of loop type, then 
$|\SL(W)|=2\gcd(1+a_1a_2a_3,1-a_1+a_1a_2).$
\item If $f$ is of fermat type, then $|\SL(W)|=\frac{2a_1a_2a_3}{\lcm(a_1,a_2,a_3)}$.
\item If $f$ is of chain+fermat type, then $|\SL(W)|=2a_3\gcd(a_1a_2,a_1-1)$ if 
$a_3$ is odd and $|\SL(W)|=a_3\gcd(a_1a_2,a_1-1)$ otherwise.
\item If $f$ is of loop+fermat type, then $|\SL(W)|=2a_3\gcd(a_1a_2-1,a_2-1)$ 
if $a_3$ is odd and $|\SL(W)|=a_3\gcd(a_1a_2-1,a_2-1)$ otherwise.
\end{itemize}
\end{pro}
\begin{proof}
In all cases we will apply Corollary \ref{sl}. We denote by 
$A$ the matrix associated to the potential $W$, by 
$w^T=(w_1^T,w_2^T,w_3^T,w_4^T)$ the weight vector of $W^T$ and 
by $d^T$ the degree of the hypersurface $W^T$ in $\IP(w^T)$.

If $f$ is of chain type, then $\det(A)=2a_1a_2a_3$. 
From the linear system 
\begin{equation}\label{e}
A^T\cdot w^T=d^Te
\end{equation} 
we obtain that
$$\frac{2a_1a_2a_3}{d^T}w^T_4=2(1-a_1+a_1a_2).$$
Let $m:=\gcd(d^T,w_4^T)$.
Observe that $m$ divides $w_2^T,w_3^T,w_4^T$ by (\ref{e}). 
Since $W$ is well-formed, this implies that $m=1$, so that
 $|\SL(W)|=\frac{\det(A)}{d^T}=2\gcd(a_1a_2a_3,1-a_1+a_1a_2)$.
The case when $f$ is of loop type is similar to the previous one.

If $f$ is of chain+fermat type, then looking at the equation of the linear system (\ref{e}) 
coming from the chain part we obtain that 
$$\frac{a_1a_2}{d^T}w^T_3=a_1-1.$$
Let $m:=\gcd(d^T,w_3^T)$. Observe that $m$ divides $w_2^T, w_3^T$,
$2w_1^T$ and $a_3w_4^T$. 
Thus, since $W$ is well-formed, $m$ is either $1$ or $2$.
Moreover, again since $W$ is well-formed, 
$m=2$ if and only if $a_3$ is even. 
If $m=1$, then $|\SL(W)|=2a_3\gcd(a_1a_2,a_1-1)$,
otherwise $|\SL(W)|=a_3\gcd(a_1a_2,a_1-1)$.
The case when $f$ is of loop+fermat type is similar.
\end{proof}

\begin{remark} The formulas given in Proposition \ref{sl4} for the chain, 
loop and fermat case can be easily generalized to 
the case of a higher dimensional well-formed potential of type $x^2=f(x_1,\dots,x_n)$.
Let $\Theta=1+\sum_{j=1}^{n-1}(-1)^{j}a_1\cdots a_j.$
In the chain case  $|\SL(W)|=2\gcd(a_1\cdots a_n,\Theta)$, in the loop case  $|\SL(W)|=2\gcd((-1)^{n-1}+a_1\cdots a_n,\Theta)$ and finally in the fermat case $|\SL(W)|=\frac{2a_1\cdots a_n}{\lcm(a_1,\dots,a_n)}$.
\end{remark}

 \subsection{Relation with Borisov's description}
In this subsection we relate the definition of transpose group with the one given in \cite{borisov}.
Let $M_0^*=N_0=\IZ^n$ and $\xi:N_0\to M_0^*,\ u\mapsto Av$. We have an exact sequence
$$0\to N_0\stackrel{\xi}{\to} M_0^* \stackrel{f}{\to}  \Aut(W)\to 0,$$
where $f(e_i)=A^{-1}e_i$. Thus $f$ induces an isomorphism $\Aut(W)\cong M_0^*/\xi(N_0)=\IZ^n/A\IZ^n$. 
The dual of $\xi$ gives the exact sequence 
$$0\to M_0\stackrel{\xi^*}{\to} N_0^*\stackrel{f^T}{\to} \Aut(W^T)\to 0,$$
where $f^T(e_i)=(A^{-1})^Te_i$,
giving isomorphisms 
$$\Aut(W^T)\cong \Ext^1(\Aut(W),\IZ)\cong N_0^*/\xi^*(M_0)=\IZ^n/A^T\IZ^n.$$ 

Let $G$ be a subgroup of $\Aut(W)$. Then there is a submodule $N=B\IZ^n$ of $M_0^*$ containing $\xi(N_0)$ such that $G\cong N/\xi(N_0)\cong B\IZ^n/A\IZ^n$. 
Observe that we can write $A=BC$, where $B,C$ are integral matrices invertible over $\IQ$. 
Consider the chain of inclusions
$$\xi(N_0)=A\IZ^n\hookrightarrow N=B\IZ^n\hookrightarrow \IZ^n=M_0^*,$$
and its dual 
$$M_0={\rm Hom}(\IZ,\IZ^n) \rightarrow N^*={\rm Hom}(B\IZ^n,\IZ)\rightarrow \xi(N_0)^*={\rm Hom}(A\IZ^n,\IZ).$$
We identify ${\rm Hom}(A\IZ^n,\IZ)$ with $\IZ^n$ via the homomorphism 
given by the dual of $\xi:N_0\to \xi(N_0)$:
$$h:{\rm Hom}(A\IZ^n,\IZ)\to \IZ^n,\  (A^T)^{-1}e_i\mapsto e_i.$$
Thus $h (N^*)=C^T\IZ^n$ and $h(M_0)=A^T\IZ^n$. 
According to Borisov's  definition
$$G^T:=N^*/M_0\stackrel{h}{\cong} C^T\IZ^n/A^T\IZ^n\stackrel{f^T}{\cong} (B^T)^{-1}\IZ^n/\IZ^n,$$
which agrees with the definition given in the first section.

\subsection{$G$ and $G^T$ as orthogonal groups}
Consider the bilinear form $b:\IQ^n\times \IQ^n\to \IQ$ given by $b(u,v)=u^TAv$.
Observe that this induces a bilinear form
$$\bar b: \Aut(W^T)\times \Aut(W) \to \IQ/\IZ,$$
where we recall that $\Aut(W^T)=(A^T)^{-1}\IZ^n/\IZ^n$ and $\Aut(W)=A^{-1}\IZ^n/\IZ^n$.
In fact $\bar b$ is well defined since, if $u,v\in \IZ^n$, then 
$$b(u,A^{-1}v)=u^TAA^{-1}v=u^Tv\in \IZ,\quad b((A^T)^{-1}u,v)=u^TA^{-1}Av=u^Tv\in \IZ.$$
Let $G=C^{-1}\IZ^n/\IZ^n$ be a subgroup of $\Aut(W)$.  We show that $G^T$ is the orthogonal  of $G$ with respect to $\bar b$. 
\begin{lemma}\label{ortog}
$$G^T=\{x\in \Aut(W^T): \bar b(x,y)=0,\ \forall y\in G\}.$$
\end{lemma}
\begin{proof} Let $u\in \IZ^n$ and $x=(A^T)^{-1}u\in \Aut(W^T)$.
We have that
$$\bar b(x,C^{-1}v)=u^TA^{-1}AC^{-1}v=u^TC^{-1}v=0$$
for all $v\in \IZ^n$, if and only if $(C^T)^{-1}u\in \IZ^n$, i.e. $(A^T)^{-1}u\in (A^T)^{-1}C^T\IZ^n=(B^T)^{-1}\IZ^n$, which means that $x\in G^T$.
\end{proof}
This remark relates our definition of transpose group with the one given in \cite{ebelingzade}.
\subsection{Relation with Krawitz's description}
According to Krawitz's definition in \cite{krawitz} the transpose group is
$$G^T=\left\{\prod_{j=1}^n(\rho_j^T)^{m_j}: \prod_{j=1}^nx_j^{m_j} \mbox{ is } G\mbox{-invariant}\right\}.$$
Observe that $\prod_{j=1}^n(\rho_j^T)^{m_j}$ corresponds, in $\Aut(W^T)$, to $\sum_jm_j(A^T)^{-1}e_j=(A^T)^{-1}m$, where $m=(m_1,\dots,m_n)\in \IZ^n$.
Moreover, $\prod_{j=1}^nx_j^{m_j}$ is $G$-invariant if and only if $\sum_j m_ja_j\in \IZ$ for all $a=(a_1,\dots,a_n)\in G$, i.e.
$$\sum_j m_ja_j=(m^TA^{-1})Aa=\bar b((A^T)^{-1}m,a)=0.$$ 
Thus $G^T$ is the orthogonal complement of $G$ with respect to $\bar b$, in agreement with Lemma \ref{ortog}.

\section{The Dolgachev-Voisin mirror symmetry for K3 surfaces}
\subsection{K3 surfaces with non-symplectic involutions}
We briefly recall the classification theorem for non-symplectic involutions on K3 surfaces given by Nikulin in \cite[\S 4]{nikulinfactor} and \cite[\S 4]{nikulinlob}. 
Let $X$ be a K3 surface and $\iota$ be non-symplectic involution of $X$.
The local action of $\iota$ at a fixed point is of type:
$$\left(\begin{array}{cc}
1 &0\\
0& -1
\end{array}
\right),$$
so that the fixed locus $X^{\iota}$ is the disjoint union of smooth curves and there are no isolated fixed points.
The invariant lattice: 
$$
H^2(X,\IZ)^{+}:=\{x\in H^2(X,\IZ)\,|\,\iota^* x=x\}
$$ 
is $2$-elementary, i.e the discriminant group $(H^2(X,\IZ)^+)^\vee/H^2(X,\IZ)^+$ is isomorphic to $(\IZ/2\IZ)^{\oplus a}$ for some non negative integer $a$. According to Rudakov-Shafarevich in \cite{RS}, the isometry class of such lattice is determined by the invariants $r, a$ and $\delta$, where $r=\rank H^2(X,\IZ)^{+}$ and $\delta\in \{0,1\}$ is $0$ if and only if $x^2\in\IZ$ for any $x\in (H^2(X,\IZ)^{+})^\vee$.   Equivalently, by \cite[\S 4]{nikulinfactor}, $\delta=0$ if and only if 
the class of the fixed locus of $\iota$ is divisible by two in $H^2(X,\IZ)$.
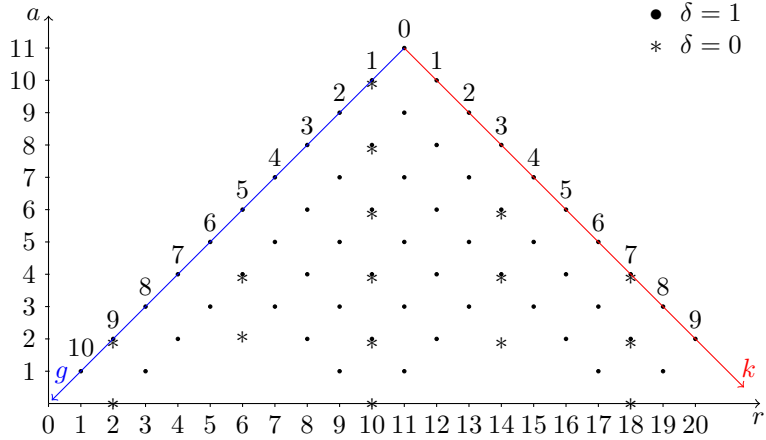
\begin{figure}[h]
$$\begin{array}{cccccccccccccccccccccccr}
 &\ &\ &&&& &&\ &\ &&&&&&&&&&&&& \bullet\ \ \delta=1\\
&\ &\ &&&& &&\ &\ &&&&&&&&&&&&&\ast\ \ \delta=0
\end{array}$$
\vspace{-1.1cm}

\begin{tikzpicture}[scale=.43]
\filldraw [black] 
(1,1) circle (1.5pt)  node[below=-0.55cm]{10}
(2,0) node[below=-0.20cm]{*} 
(2,2) circle (1.5pt) node[below=-0.5cm]{9} node[below=-0.15cm]{*}
 (3,1) circle (1.5pt)
 (3,3) circle (1.5pt)node[below=-0.5cm]{8}
 (4,2) circle (1.5pt)
(4,4) circle (1.5pt)node[below=-0.5cm]{7}
(5,3) circle (1.5pt)
(5,5) circle (1.5pt)node[below=-0.5cm]{6}
(6,4) circle (1.5pt)node[below=-0.15cm]{*}
(6,2)    node[below=-0.23cm]{*}
(6,6) circle (1.5pt)node[below=-0.5cm]{5}
(7,3) circle (1.5pt)
(7,5) circle (1.5pt)
(7,7) circle (1.5pt)node[below=-0.5cm]{4}
(8,2) circle (1.5pt)
(8,4) circle (1.5pt)
(8,6) circle (1.5pt)
(8,8) circle (1.5pt)node[below=-0.5cm]{3}
(9,1) circle (1.5pt)
(9,3) circle (1.5pt)
(9,5) circle (1.5pt)
(9,7) circle (1.5pt)
(9,9) circle (1.5pt)node[below=-0.5cm]{2}
(10,0)  node[below=-0.20cm]{*}
(10,2) circle (1.5pt)node[below=-0.15cm]{*}
(10,4) circle (1.5pt)node[below=-0.15cm]{*}
(10,6) circle (1.5pt)node[below=-0.15cm]{*}
(10,8) circle (1.5pt)node[below=-0.15cm]{*}
(10,10) circle (1.5pt) node[below=-0.5cm]{1}node[below=-0.15cm]{*}
(11,1) circle (1.5pt)
(11,3) circle (1.5pt)
(11,5) circle (1.5pt)
(11,7) circle (1.5pt)
(11,9) circle (1.5pt)
(11,11) circle (1.5pt) node[below=-0.5cm]{0}
(12,2) circle (1.5pt)
(12,4) circle (1.5pt)
(12,6) circle (1.5pt)
(12,8) circle (1.5pt)
(12,10) circle (1.5pt) node[below=-0.5cm]{1}
(13,3) circle (1.5pt)
(13,5) circle (1.5pt)
(13,7) circle (1.5pt)
(13,9) circle (1.5pt) node[below=-0.5cm]{2}
(14,2)  node[below=-0.15cm]{*}
(14,4) circle (1.5pt)node[below=-0.15cm]{*}
(14,6) circle (1.5pt)node[below=-0.15cm]{*}
(14,8) circle (1.5pt) node[below=-0.5cm]{3}
(15,3) circle (1.5pt)
(15,5) circle (1.5pt)
(15,7) circle (1.5pt) node[below=-0.5cm]{4}
(16,2) circle (1.5pt)
(16,4) circle (1.5pt)
(16,6) circle (1.5pt) node[below=-0.5cm]{5}
(17,1) circle (1.5pt)
(17,3) circle (1.5pt)
(17,5) circle (1.5pt)node[below=-0.5cm]{6}
(18,0)  node[below=-0.2cm]{*}
(18,2) circle (1.5pt)node[below=-0.15cm]{*}
(18,4) circle (1.5pt) node[below=-0.5cm]{7}node[below=-0.15cm]{*}
(19,1) circle (1.5pt)
(19,3) circle (1.5pt) node[below=-0.5cm]{8}
(20,2) circle (1.5pt) node[below=-0.5cm]{9}
 ; 
\draw plot[mark=*] file {data/ScatterPlotExampleData.data};
\draw[->] (0,0) -- coordinate (x axis mid) (22,0);
    \draw[->] (0,0) -- coordinate (y axis mid)(0,12);
    \foreach \x in {0,1,2,3,4,5,6,7,8,9,10,11,12,13,14,15,16,17,18,19,20}
        \draw [xshift=0cm](\x cm,0pt) -- (\x cm,-3pt)
         node[anchor=north] {$\x$};
          \foreach \y in {1,2,3,4,5,6,7,8,9,10,11}
        \draw (1pt,\y cm) -- (-3pt,\y cm) node[anchor=east] {$\y$};
    \node[below=0.2cm, right=4.5cm] at (x axis mid) {$r$};
    \node[left=0.2cm, below=-2.8cm, rotate=0] at (y axis mid) {$a$};
 \draw[<-, blue](0.1,0.1)-- node[below=2cm,left=2cm]{$g$} (11,11);   
 \draw[<-, red](21.5,0.5)-- node[below=2cm,right=2.1cm]{$k$} (11,11);
  \end{tikzpicture} 
\caption{Nikulin classification}\label{ord2}
\end{figure}

\begin{theorem}\cite[Theorem 4.2.2]{nikulinfactor}\label{nikulinfactor}
The fixed locus of a non-symplectic involution on a $K3$ surface is
\begin{itemize}
\item empty if $r=10$, $a=10$ and $\delta=0$,
\item the disjoint union of two elliptic curves if  $r=10$, $a=8$ and $\delta=0$,
\item the disjoint union of a curve of genus $g$ and $k$ rational curves otherwise, where
 $g=(22-r-a)/2$, $k=(r-a)/2$.
\end{itemize}
\end{theorem}
Figure \ref{ord2} shows all the values of the triple $(r,a,\delta)$ which are realized and the corresponding invariants $(g,k)$ of the fixed locus.
  
 We now assume that $X$ carries a symplectic automorphism $\sigma$ of prime order commuting with $\iota$. The minimal resolution $Y$ of $X/\langle\sigma\rangle$ is known to be a K3 surface and $\iota$ lifts to a non-symplectic involution $j$ on $Y$. The following proposition relates the invariants of $\iota$ and $j$. We denote by $\delta(\iota)$ and $\delta(j)$ the $\delta$-invariants of the invariant lattices of $\iota$ and $j$. 
 We recall that the order of a symplectic automorphism of prime order $p$ on a K3 surface 
is either $2, 3,5$ or $7$ by \cite{nikulin}.
 \begin{pro}\label{delta}
 Let $X$ be a K3 surface carrying a non-symplectic involution $\iota$ and a symplectic automorphism $\sigma$ of prime order $p>2$ commuting with $\iota$. Then $\iota$ induces a non-symplectic involution $j$ on the minimal resolution $Y$ of $X/\langle\sigma\rangle$ such that $\delta(\iota)=\delta(j)$.
 \end{pro}
 \begin{proof} Let $\pi:X\to X/\langle\sigma\rangle$ be the natural quotient map. 
 Since $\sigma$ is symplectic, the set $F$ of its fixed points is finite 
 and the quotient $X/\langle\sigma\rangle$ has singular points of type $A_{p-1}$ at $S=\pi(F)$ (see \cite{nikulin}).
We denote by $r:Y\to X/\langle\sigma\rangle$ the minimal resolution of $X/\langle\sigma\rangle$ 
and by $E^q_1,\dots,E^q_{p-1}$ the irreducible components of $r^{-1}(q)$, $q\in S$, with 
$E^q_i\cdot E^q_{i+1}=1$, $i=1,\dots,p-2$.

 Since $\iota$ commutes with $\sigma$, it induces an involution 
 on $X/\langle\sigma\rangle$ which lifts to a non-symplectic involution $j$ on $Y$. 
 Let $X^{\iota}$ be the fixed locus of $\iota$ and $Y^{j}$ be the fixed locus of $j$. 
Since the orders of $\iota$ and $\sigma$ are relatively prime, 
the fixed locus of the involution induced by $\iota$ on $X/\langle\sigma\rangle$ 
coincides with $\pi(X^{\iota})$. 
Observe that in general this is not a Cartier divisor since it passes through 
the singular points of $X/\langle \sigma \rangle$.
Taking the pull-back of $\pi(X^{\iota})$ to $Y$ we obtain the following $\IQ$-divisor:
\begin{equation}\label{e1}
r^*(\pi(X^{\iota}))=\tilde X^{\iota}+\frac{1}{p}\sum_{q\in S\cap \pi(X^{\iota})} \sum_{i=1}^{p-1} iE^q_i,
\end{equation}
where $\tilde X^{\iota}$ is the proper transform of $\pi(X^\iota)$ and it only intersects $E^q_{p-1}$ 
for any $q\in S\cap \pi(X^{\iota})$.
Since $j$ leaves invariant the exceptional divisors over 
$S\cap \pi(X^{\iota})$ and, being non-symplectic,  it only fixes smooth disjoint curves, we have
\begin{equation}\label{e2}
Y^j=\tilde X^{\iota}+\sum_{k=1}^{\frac{p-1}{2}}\sum_{q\in S\cap \pi(X^{\iota})} E^q_{2k-1}.
\end{equation}
Observe that the natural inclusions induce isomorphisms $\Cl(X)\cong \Cl(X- F)$ 
and $\Cl(X/\langle\sigma\rangle)\cong \Cl(X/\langle\sigma\rangle- S)$ 
since $F$ and $S$ have codimension two.
Moreover $\pi_0:X- F\to X/\langle\sigma\rangle- S$ is an unramified covering.
Now assume that $\alpha:=[X^{\iota}]$ is divisible by  two in $H^2(X,\IZ)$ or, equivalently 
in $\Cl(X)$, and let $\beta:=[\pi_0(X^{\iota})]$.
Then by projection formula ${\pi_0}_*(\alpha)={\pi_0}_*\pi_0^*(\beta)=p\beta$ 
is also divisible by two.
From equalities (\ref{e1}) and (\ref{e2}) we get:
$$r^*(p\beta)\equiv  [Y^j]\ (\rm{mod}\, 2).$$
Thus $[Y^j]$ is divisible by two. Conversely, if $[Y^j]$ is divisible by two, the same is true for $r^*(p \beta)$ 
 by the previous congruence.  
By projection formula $\beta$ is divisible by two in  $\Cl(X/\langle\sigma\rangle)$, 
thus the same is true for $\pi_0^*(\beta)=\alpha$.
 \end{proof}

 \begin{remark}
Equation \ref{e1} can be checked by means of a local computation at an $A_{p-1}$ singularity, 
for example computing  the pull-back of the invariant divisors 
of the toric variety associated to the fan with rays $(p,1-p),(0,1)$ by means of  MAGMA \cite{magma}.

 \end{remark}


\subsection{The Dolgachev-Voisin mirror symmetry}\label{DVmirror}
Let $M$ be an even non-degenerate lattice of signature $(1,\rho-1)$, $1\leq \rho\leq 19$.
\begin{definition}
An {\it $M$-polarized} K3 surface is a pair $(X,j)$ where $X$ is a K3 surface and $j:M\hookrightarrow \Pic(X)$ is a primitive lattice embedding.
\end{definition}
 Dolgachev in \cite{dolgachevmirror} constructs a (coarse) moduli space $\KK_M$ parametrizing $M$-polarized K3 surfaces, which has dimension $20-\rho$. Assume now that 
 $$
 M^{\perp}\cap H^2(X,\IZ)\cong U\oplus\bar{M}, 
 $$
 where $U$ is a copy of the hyperbolic plane. As described in \cite{dolgachevmirror} one can define the {\it mirror moduli space} of $\KK_M$ as the 
 moduli space $\KK_{\bar{M}}$ of $\bar{M}$-polarized K3 surfaces: one can use the primitive embedding $\bar{M}
 \hookrightarrow M^{\perp}\subset H^2(X,\IZ)$ to get a primitive even non-degenerate sublattice of signature $(1,(20-\rho)-1)$ of the K3 lattice $U^3\oplus E_8(-1)^2$. Observe that for generic K3 surfaces $X_M\in\KK_M$ and $X_{\bar{M}}\in\KK_{\bar{M}}$ 
 we have
 $$
 \dim \KK_M=20-\rho= \rank \Pic(X_{\bar{M}}),\quad  
 \dim \KK_{\bar{M}}=\rho=\rank \Pic(X_M).
 $$
 We now consider the special case when $X$ is a K3 surface admitting a non-symplectic involution
 and $M=H^2(X,\IZ)^{+}$. We denote the anti-invariant lattice by 
 $$H^2(X,\IZ)^{-}:=(H^2(X,\IZ)^{+})^{\perp}\cap H^2(X,\IZ).$$
 \begin{pro}\cite[Lemma 2.5, \S 2.3]{voisin}
 Assume that $(r,a,\delta)\not= (14,6,0)$ and $g\geq 1$. Then:
 \begin{itemize}
 \item $H^2(X,\IZ)^{-}\cong U\oplus \bar{M}$;
 \item  the generic K3 surface $X_{\bar{M}}\in \KK_{\bar{M}}$ has a non-symplectic involution;
\item if  $X_M\in\KK_M$ has invariants $(r,a,\delta)$ then the invariants of $X_{\bar{M}}\in\KK_{\bar{M}}$ are ${(20-r,a,\delta)}$.
 \end{itemize}
 \end{pro}
\begin{remark}\text{}
\begin{itemize}
\item In Figure \ref{ord2} one can see the mirror couples making a reflection with respect to the axis through
 $r=10$ and $1\leq g\leq 10$ and deleting the axis with $g=0$ and the point $(r,a,\delta)=(14,6,0)$.
\item Since K3 surfaces with a non-symplectic involution are projective
 the invariant lattice contains an ample class. One can then consider
instead of $\KK_M$ the moduli space $\KK_M^a$ of ample $M$-polarized K3 surfaces and do the same 
construction of mirror moduli spaces as above \cite{dolgachevmirror}.
\end{itemize}
\end{remark}

\section{K3 surfaces in weighted projective spaces with non-symplectic involutions}
In this section we will consider K3 surfaces obtained as desingularizations of 
hypersurfaces of the following type in some weighted projective space:
\begin{equation}\label{k3}
W(x,y,z,w)=x^2-f(y,z,w)=0.
\end{equation}
Observe that any such surface carries the non-symplectic involution $\iota: x\mapsto -x$.
We will describe their singularities and we will explain how to 
compute the triple of invariants $(r,a,\delta)$ of $\iota$ introduced in section 4. 
We recall that $h_{ij}:=\gcd(w_i,w_j)$.
\begin{lemma}\label{sing}
Let $W$ be a quasismooth and  Gorenstein hypersurface in 
 $\IP(w_1,w_2,w_3,w_4)$ defined by 
an invertible potential as in (\ref{k3}).
 Then the singular points of $W$ are Du Val singularities 
 of type $A_k$ and can only occur at the vertices 
 $P_2,P_3,P_4$ or along the edges $P_iP_j$ with $1\leq i,j\leq 4$.
 More precisely:
 \begin{itemize}
 \item[a)]  if $w_i>2$, $i=2,3,4$, then $P_i\in W$  if and only if $w_i\not|d$ 
 and in this case it 
 is a singular point of type $A_{w_i-1}$;
 \item[b)] if $i,j>1$, $w_i,w_j>2$ and $h_{ij}>1$, then $W$ intersects $P_iP_j-\{P_i,P_j\}$ 
 at $\lfloor \frac{d h_{ij}}{w_iw_j} \rfloor$ singular points of type $A_{h_{ij}-1}$;
 \item[c)] if $i,j>1$, $w_i=2$ and $w_i|w_j$, then $W$ intersects $P_iP_j-\{P_j\}$ 
 at $\lfloor \frac{d}{w_j} \rfloor$ singular points of type $A_{1}$;
  \item[d)] if $h_{1i}>1$, then $P_i\not\in W$ and $W$ intersects $P_1P_i$ 
  at two points if $w_i|w_1$ and at one point otherwise.
  In both cases the intersection points are singularities of type $A_{h_{1i}-1}$.
 \end{itemize}
\end{lemma}
\begin{proof} We recall that, since $W=0$ is quasismooth, then  
it is well-formed and $W$ is non-degenerate. This implies that the singularities of $W$ 
can only appear along the vertices $P_i$ or the edges $P_iP_j$, where 
the singular points of  the ambient projective space occur.
Since $W$ is Gorenstein and quasismooth, then it has only cyclic, canonical   
singularities \cite{cortigo}, i.e. its singular points are Du Val of type $A_k$. 
More precisely, a vertex $P_i\in W$ is a singular point of type 
$A_{w_i-1}$   and 
an intersection point of $W$ with an edge 
$P_iP_j$  (outside of the vertices) is a singularity of type 
$A_{h_{ij}-1}$
where $h_{ij}=\gcd(w_i,w_j)$.

We first observe that $P_1=(1,0,0,0)\not\in W$.
Moreover, if $w_2$ does not divide the degree $d$ of $W$, then clearly $P_2=(0,1,0,0)\in W$. 
Now assume that $w_2>2$, $w_2$ divides $d$ and $P_2\in W$. 
Thus $f$ is of one of the following types 
up to a change of coordinates:
$$y^az+z^bw+w^c,\quad  y^az+z^bw+w^cy,\quad y^az+z^b+w^c,\quad y^az+z^by+w^c.$$
In the first case, the linear system
$$\left\{
\begin{array}{l}
2w_1=d\\
aw_2+w_3=d\\
bw_3+w_4=d\\
cw_4=d
\end{array}
\right.
$$
implies that $w_2$ divides $w_3,w_4$, contradicting the fact that $W$ is well-formed.
Similarly, the second case does not occur.
In the third case the analogous linear system gives that $w_2$ divides $w_3$ and $2w_1$.
Since $W$ is well-formed, this implies that $w_2=2$, giving a contradiction. The last case is similar.
Thus, in case $w_2$ divides $d$ and $w_2>2$, then $P_2\not\in W$. This proves a).

If $w_2=2$, then $P_2$ can be either on $W$ or not, but in any case 
its singular type is the same of the generic point of the 
singular edges containing it. 
A similar discussion holds for $P_3$ and $P_4$.

Now assume that $w_2,w_3>2$ and $h_{23}=\gcd(w_2,w_3)>1$. 
The number of intersection points 
of $W$ with the edge $P_2P_3\cong 
\IP(w_2,w_3)\cong \IP(v_2,v_3)$, 
where $v_i:=\frac{w_i}{h_{23}}$,
only depends on the weights.
In fact, assume that $P_2,P_3\in W$ and let $d':=d/h_{23}$.
Then $\bar f:=\frac{f(0,y,z,0)}{yz}$ is of the form
$$\bar f=y^{\frac{d'-v_2-v_3}{v_2}}+z^{\frac{d'-v_2-v_3}{v_3}}.$$
Thus we obtain that $\bar f/z^{\frac{d'-v_2-v_3}{v_3}}=(\frac{y^{v_3}}{z^{v_2}})^{\frac{d'-v_2-v_3}{v_2v_3}}+1$,
so that, since $y^{v_3}/z^{v_2}$ is an affine coordinate,  
$f$ has $\frac{d'-v_2-v_3}{v_2v_3}$ distinct points outside of the vertices.
Observe that such number equals $\lfloor \frac{dh_{23}}{w_2w_3}\rfloor$.
In fact 
$$
\frac{dh_{23}}{w_2w_3}-\frac{d'-v_2-v_3}{v_2v_3}=\frac{a_2+a_3}{a_2a_3},
$$
where $a_i:=w_i/h_{23}$, $i=2,3$.
If $a_2>2$ and $a_3>2$, then the right hand side is clearly smaller than one.
Otherwise, since $a_2$ and $a_3$ are relatively prime, one of them 
would be equal to one, for example $w_2=h_{23}$.
 This contradicts the fact that $w_2$ does not divide $d$, since $P_2\in W$, by 
 the first point in the proof.

The case when either $P_2$ or $P_3$, or both, do not belong to $W$ is similar 
(see also the proof of \cite[Lemma I.6.3]{fletcher}). This proves b) and c).

Finally, assume that $h_{12}=\gcd(w_1,w_2)>1$.
It can be easily proved that $P_2\not\in W$ since $W$ is well-formed.
In this case the intersection of $f$ with the edge 
$P_1P_2$ is given by the solutions of an equation of type 
$x^2-y^a=0$ in $\IP(w_1,w_2)\cong \IP(v_1,v_2)$,
where $v_i:=w_i/h_{12}$. Observe that $w_2$ divides $2w_1$ in this case.
The previous equation has two solutions if $w_2$ divides $w_1$,
since in this case $\frac{x}{y^{1/2}}$ is a coordinate.
Otherwise, $w_2=2\gcd(w_1,w_2)$,  and 
the equation has a unique solution since a 
coordinate function is $\frac{x^2}{y^a}$. This proves d).
\end{proof}

\begin{example}\label{9432} Consider a quasismooth and invertible potential $W(x,y,z,w)=x^2-f(y,z,w)$ defining a 
degree $18$ hypersurface in $\IP(9,4,3,2)$.
Observe that $W$ has a singular point of type $A_3$ at $P_2$ since $w_2$ does not divide $18$. 
On the other hand, $P_3\not\in W$ since $w_3>2$ and $w_3$ divides $18$. 
The point $P_4$ can be either in $W$ or not, depending on $f$.
The surface $W$ intersects the edge $P_1P_3$ in $2$ points, exchanged by $\iota$, 
since $w_3$ divides $w_1$. These are singularities of type $A_2$.
Finally, we consider the intersection of $W$ with the edge $P_2P_4\cong \IP(2,1)$.
Observe that in this case we obtain a degree $9$ equation $f'(y,w)=0$.
If $P_4\in W$, then $W$ intersects the edge in $\frac{9-1}{2}=4$ points outside the vertices.
Otherwise, if $P_4\not\in W$, then $W$ intersects the edge in $\frac{9-1-2}{2}=3$ points outside the vertices.
In any case we have exactly one singular point of type $A_3$ and 
four singular points of type $A_1$ along the edge.
\end{example}

Let $W\subset \IP(w_1,w_2,w_3,w_4)$ be
defined by an equation of type  (\ref{k3})
and let $\gamma:X\to W$ be its minimal resolution.
We will denote by  $\iota $ both the involution 
$x\mapsto -x$ on $W$ and the involution induced by this on $X$.
We will consider the following commutative diagram, 
where $\IP:=\IP(w_2,w_3,w_4)$ and $\gamma_1:\widetilde \IP\to \IP$ is its minimal resolution, 
$\pi$ and $\tilde \pi$ are the quotients by $\iota$ and
 $\gamma_2$ is the blow up of $\widetilde \IP$ at the singular 
 points of the pull-back of the 
 branch locus of $\pi$. 
 $$
 \xymatrix{
 X\ar[rr]^\gamma\ar[d]_{\tilde \pi} & & W\ar[d]^{\pi} \\
 Y\ar[r]^{\gamma_2} & \widetilde \IP \ar[r]^{\gamma_1} & \IP
 }
 $$
 \begin{lemma} \label{invlatt}
 Assume that the fixed locus of $\iota$ on $X$ is of the form
 $$C\cup E_1\cup\cdots \cup E_k,$$
 where $g(C)\geq 1$ and $g(E_i)=0$.
 Then the invariant lattice $H^2(X,\IZ)^+$ is generated by $\tilde\pi^*\Pic(Y)$ and the classes 
 of $E_1,\dots,E_k$.
 \end{lemma}
\begin{proof}
We first observe that $\tilde\pi^*\Pic(Y)\otimes \IQ= H^2(X,\IQ)^+$.
In fact, if $x\in H^2(X,\IZ)$, then $\tilde\pi^*\tilde \pi_*(x)=x+\iota^*x$.
This proves that $\tilde\pi^*\Pic(Y)\subset H^2(X,\IZ)^+$ and that 
$2x\in \tilde\pi^*\Pic(Y)$ for any $x\in H^2(X,\IZ)^+$.
Let $r=\rk H^2(X,\IZ)^+=\rk \Pic(Y)$.

Since $\iota$ is a non-symplectic involution, then $Y=X/\langle\iota\rangle$ is a smooth rational surface.
In particular $\Pic(Y)=H^2(Y,\IZ)$ is a unimodular lattice, thus the determinant of the lattice
$\tilde\pi^*\Pic(Y)=\Pic(Y)(2)$ equals $\pm 2^r$. 
Let $x_i$ be the class of  $E_i$, then 
$x_i\not\in \tilde\pi^*\Pic(Y)$ and $2x_i=\tilde\pi^*(y_i)$, 
where $y_i$ is the class of its image in $Y$.
The lattice generated by $\tilde\pi^*\Pic(Y)$ and $x_1,\dots,x_k$ has determinant 
$\pm 2^r/2^{k}=\pm 2^{r}/2^{r-a}=\pm 2^a$ by Theorem \ref{nikulinfactor}, thus it coincides with $H^2(X,\IZ)^+$.
 \end{proof}

In order to compute the triple $(r,a,\delta)$ for the lattice $H^2(X,\IZ)^+$ we follow these steps:
\begin{itemize}
\item we identify the irreducible components of the fixed locus of $\iota$ in $W$ and the number of 
their intersection points: $W^{\iota}$ always contains the curve $C$ defined by $x=0$ and 
possibly one more curve, 
defined by the vanishing of another coordinate;
\item denoting by $B$ the branch locus of $\pi$, we identify the singularities of $\IP$ on $B$;
\item we compute $r=\rk \Pic(Y)$ as the sum of the Picard number of $\widetilde\IP$ 
with the number $s$ of singular points of $\gamma_1^*B$;
\item we recall that $a=22-r-2g$ by Theorem \ref{nikulinfactor}, thus to obtain $a$ it is enough to compute the genus of the curve $x=0$ by means of the formula given in \S 2.1;

\item in order to identify $\delta$, we compute the invariant lattice of $X$ as follows: we observe that 
$\tilde\pi^*\Pic(Y)=M(2)\oplus (-2)^s$, where $M=\Pic(\tilde\IP)$, and we add to this lattice the classes of the rational curves in the ramification locus of $\tilde\pi$ (their classes can be computed by looking at their intersection with the generators of $\tilde\pi^*\Pic(Y)$).
\end{itemize}
The invariant $r$ can also be computed as follows: let ${\rm Exc}(\gamma)$ be the lattice generated by the exceptional divisors of $\gamma$.  Then $r=1+ \rk{\rm Exc}(\gamma)^{\iota}$, where $1=\rk H^2(W,\IZ)^{+}=\rk \Cl(\IP)$ and $\rk{\rm Exc}(\gamma)^{\iota}$ equals the number of $\iota$-orbits in the exceptional locus of  $\gamma$.
\begin{remark}\label{ind}
We observe that the triple $(r,a,\delta)$ only depends on the weight vector $w=(w_1,w_2,w_3,w_4)$.
In fact, the configuration of the irreducible components of $W^{\iota}$ (i.e. their number and mutual intersections) only depends on $w$, and the same holds for the singularities of $W$ by Lemma \ref{sing}.

\end{remark}



\begin{example}\label{tor} We now compute the triple $(r,a,\delta)$ for the surface $W$ in Example \ref{9432}.
The projective plane $\IP=\IP(4,3,2)\cong \IP(2,3,1)$ has a singular point 
of type $A_1$ at $(1,0,0)$ and one of type $A_2$ at $(0,1,0)$.
Its minimal resolution is a toric variety $\tilde \IP$ whose fan has six rays:
$$r_1=(-1,  1),\ r_2=( 0, -1),\ r_3=( 2,  1),\ r_4=( 1,  1),\ r_5=( 0,  1),\ r_6=( 1,  0),$$
where $r_6$ corresponds to the exceptional divisor over the $A_1$ singularity, 
$r_4,r_5$ to the two components of the exceptional divisor over the $A_2$ singularity 
and $r_3$ to the proper transform of the line through the two singular points of $\IP$.
A basis of  $\Pic(\tilde\IP)$ is given by the classes $v_1,v_2,v_3,v_4$ of the last four rays.
With respect to this basis, the classes of the six rays are given by the columns of the following matrix
$$\left(
\begin{matrix}
0 & 1 & 0 & 0 & 1 & 0\\
1 & 1 & 0 & 0 & 0 & 1\\
1 & 2 & 0 & 1 & 0 & 0\\
    2 & 3 & 1 & 0 & 0 & 0
    \end{matrix}
    \right)
   .$$
An easy computation shows that the Picard lattice 
of $\tilde\IP$ has intersection matrix:
$$
M:=\left(
\begin{matrix}
-1 & 1 & 0 & 1 \\
1 & -2 & 1 & 0 \\
0 & 1 & -2 & 0 \\
1 & 0 & 0 & -2
    \end{matrix}
    \right)
$$
The branch locus $B$ of $\pi$ is the union of the curves 
$$B_1: f(y,z,w)=0\qquad B_2:\ z=0.$$
Observe that $B_1$ and $B_2$ intersect at $(1,0,0)$ and at $4=\lfloor \frac{d\cdot h_{24}}{w_2w_4}\rfloor$ 
other points (see the second point in Lemma \ref{sing}). 
The pull-back $\gamma_1^*B$ in $\widetilde\IP$ has three irreducible 
components: the proper transforms $\tilde B_1,\tilde B_2$ and the exceptional divisor 
$E$ over the singular point, with $\tilde B_1\cdot  \tilde B_2=4$ and  $\tilde B_i\cdot E=1,\ i=1,2$.
The surface $Y$ is the blow-up of $\widetilde \IP$ at the six singular points of $\gamma_1^*B$, 
thus its Picard lattice has intersection matrix $M\oplus (-1)^6$. We still denote by $\tilde B_1,\tilde B_2$ 
the proper transforms of the curves in $Y$.

Let $v_5,\dots,v_8$ be the classes of the exceptional divisors over the points in 
$\tilde B_1\cap \tilde B_2$, $v_9$ the one over $E\cap \tilde B_1$ and $v_{10}$ over $E\cap \tilde B_2$. 
We now compute $H^2(X,\IZ)^{+}$:  this is obtained by adding to the lattice $\tilde\pi^* \Pic(Y)=M(2)\oplus (-2)^6$ 
the classes of the rational curves in the fixed locus, in this case 
$\tilde\pi^*([\tilde B_2])/2=(3v_1+2v_2+v_3+v_4-v_5-v_6-v_7-v_8-v_{10})/2$ and $\tilde\pi^*([E])/2=(v_4-v_9-v_{10})/2$.
 Computing the discriminant group of the lattice by means of a computer algebra program, 
 we see that $\delta=1$.
 Thus $(r,a,\delta)=(10,6,1)$.
\end{example}
\section{The Berglund-H\"ubsch-Chiodo-Ruan mirror symmetry for K3 surfaces}\label{proof}

In this section we prove Theorem \ref{main} by means of a classification of K3 surfaces defined by 
a non-degenerate invertible potential of the form $W(x,y,z,w)=x^2-f(y,z,w)$ in some weighted projective space. The possible decompositions of the polynomial $f(y,z,w)$ as a sum of atomic types are the following, up to a permutation of the variables $y,z,w$: \\
\begin{enumerate}[i)] \itemsep 0.4cm
\item chain: $W_c=x^2-y^{a_1}z+z^{a_2}w+w^{a_3}$,
\item loop: $W_l=x^2-y^{a_1}z+z^{a_2}w+w^{a_3}y$,
\item fermat: $W_f=x^2-y^{a_1}+z^{a_2}+w^{a_3}$,
\item chain+fermat: $W_{cf}=x^2-y^{a_1}z+z^{a_2}+w^{a_3}$,
\item loop+fermat: $W_{lf}=x^2-y^{a_1}z+z^{a_2}y+w^{a_3}$.\\
\end{enumerate}
  
Borcea in \cite[Tables 1, 2, 3]{Borcea} and Yonemura in \cite[Table 2.2]{yonemura} classified equations of K3 surfaces in weighted projective 3-spaces, but these are not always 
of Delsarte type. Thus our first aim is to identify which weights $w$ admit a quasi-homogeneous equation of type $x^2=f(y,z,w)$, where $f$ is as in i), ii), iii), iv) or v), and then to write the possible equations for a given weight.
The result of this classification is contained in the first two columns of Tables \ref{K3fermatmirror}, \ref{K3loopmirror},  \ref{K3lfmirror},  \ref{K3cfmirror}, \ref{K3chainmirror}.


We briefly explain the notation in the tables. In the first column we number the K3 surface $W:\ x^2-f(y,z,w)=0$ following \cite{Borcea} and we put in parenthesis the number corresponding to the transposed K3 surface $W^T$. 
In the fourth column appear the Nikulin's invariants $(r,a,\delta)$ of the involution $\iota$ on  the resolution $X$ of $W$, computed as explained in section \S 5.
In the last two columns we compute the orders of the groups $\SL(W)$ (by means of Proposition~\ref{sl4}) and $J_W$.

\subsection{Trivial $\widetilde{\SL(W)}$}
Observe that in every case, except for the ones marked with $*$, we have that the group  $\SL(W)/J_W$ is trivial,
so that $W$ and $W^T$ are BHCR-mirror of each other. 
As the tables show, for such pairs the invariants $(r,a,\delta)$ are mirror in the sense of Dolgachev-Voisin (i.e. they are $(r,a,\delta)$ and $(20-r,a,\delta)$), thus 
the theorem is proved in these cases.

\begin{example}  We consider the case of the weight vector $w=(5,3,1,1).$
In order to determine which $f$ can appear in this weight, we need to solve the linear system
$$Aw=10e,$$
where $e$ is the column vector with all entries equal to $1$ and 
$A$ is the matrix associated to one of the potentials 
$W_c,W_l,W_f,W_{cf}, W_{lf}$  (and the ones obtained from them by a coordinate change).

\textbf{No.\,3a and 28  in Table \ref{K3chainmirror}.} If $A$ is associated to the potential $W_c$, 
the only solution is $(a_1,a_2,a_3)=(3,9,10)$, which gives the  
surface No.\,3a in Table \ref{K3chainmirror}.
This has only one $A_2$ singularity and $g=9$ so that $(r,a,\delta)=(3,1,1)$ (here $\delta$ is uniquely determined, see Figure \ref{ord2}). 
The surface $W^T$ is No. 28 in Table \ref{K3chainmirror}. Its configuration of singular fibers is $A_1+A_3+A_4+A_8$, so that  $(r,a,\delta)=(17,1,1)$. 
By Proposition~\ref{sl4} we find that $J_{W^T}=\SL(W^T)$, so $W$ and $W^T$ are BHCR-mirror and belong to Dolgachev-Voisin  mirror families.

\textbf{No.\,3b and 5 in Table \ref{K3chainmirror}.}
 If we consider the potential $W_c$ with the variables $y$ ad $z$ exchanged, we find another solution with  
$(a_1,a_2,a_3)=(7,3,10)$. This gives case No.\,3b in Table \ref{K3chainmirror}, which has again $(r,a,\delta)=(3,1,1)$. 
The surface $W^T$ is No.~5 in Table  \ref{K3chainmirror} and has $3A_1+A_3$ singular points invariant for $\iota$, so that $(r,a,\delta)=(7,3,1)$. 
Here $\SL(W)/J_W\cong \IZ/3\IZ$, so that $W$ and $W^T$ are not BHCR-mirror.

\textbf{No.\,3 and 23 in Table \ref{K3loopmirror}.}
If $A$ is of loop type then the only solution is $(a_1,a_2,a_3)=(3,9,7)$, which gives the surface No.\,3 in Table \ref{K3loopmirror}. This surface has again $(r,a,\delta)=(3,1,1)$ (see Remark \ref{ind}).
The surface $W^T$ is No.\,23  in Table \ref{K3loopmirror}.
Here again $J_{W}=\SL(W)$, so that $W$ and $W^T$ are BHCR-mirror and belong to Dolgachev-Voisin  mirror families.

\textbf{No.\,3 and 18 in Table \ref{K3cfmirror}.}
In the chain+fermat case we obtain as a unique solution $(a_1,a_2,a_3)=(3,10,10)$, which gives No.\,3 in Table \ref{K3cfmirror}. The surface $W^T$ is given by No. 18 in the same table, but in this case 
$\SL(W)/J_W\cong \IZ/2\IZ$, so that $W$ and $W^T$ are not BHCR-mirror.

\textbf{No.\,3 in Table \ref{K3lfmirror}.}
We obtain the solution $(a_1,a_2,a_3)=(3,7,10)$ in the loop+fermat case. Here $W=W^T$ and 
$\SL(W)/J_W\cong \IZ/4\IZ$.

We will discuss the last cases  in the next subsection.
 \end{example}
\subsection{Non trivial $\widetilde{\SL(W)}$}\label{group}
In this case the BHCR-mirror pairs are given by the minimal resolutions of $W/\tilde G$ and $W^T/\tilde G^T$,
where $\tilde G^T=\SL(W^T)/J_{W^T}$  by Proposition \ref{grouprop}.
We recall that, by Proposition \ref{actions}, the group $\tilde G$ acts symplectically on $W$ and its minimal resolution $X$. 
Moreover, since it is finite and abelian it appears in the list of the $15$ possible finite symplectic abelian groups given by Nikulin in \cite{nikulin}.
Since the involution $\iota$ commutes with $\tilde G$ (which is generated by diagonal automorphisms), 
then $\iota$ clearly induces a non-symplectic involution $j$ on $X/\tilde G$ and on its minimal resolution $Y$.
We are thus interested in computing the triple $(r,a,\delta)$ for such involution on $Y$.
We have a commutative diagram:

\begin{eqnarray}\label{diagpotential}
\xymatrix{
  &X\ar[r]^-{\gamma}\ar[d]&W\ar[d]^q\\
Y\ar[r]^{\eta_2}&X/\tilde{G}\ar[r]^{\eta_1}&W/\tilde{G}
}
\end{eqnarray}
where we still denote by $\tilde{G}$ its lifting to $X$, $\eta_2$ is the minimal resolution of $X/\tilde G$ and  
$\eta_2\circ\eta_1$ the minimal resolution of $W/\tilde{G}$, whose singular locus is the image of the singular locus of $W$ and of the points with non trivial stabilizer for $G$.
The rank $r$ of the invariant lattice $H^2(Y,\IZ)^j$ equals $1$ plus the number of $j$-orbits of the exceptional locus in $Y$ and the curve of maximal genus in $\Fix(j)$ is isomorphic to $q(C)$, thus its genus $g$ can be computed by means of the Riemann-Hurwitz formula.
Finally $a$ can be computed by means of the formula in Theorem \ref{nikulinfactor}.

If $\tilde G$ is cyclic of prime order $p>2$, then the invariant $\delta$ of $(Y,j)$ equals the one of $(X,\iota)$ by Proposition \ref{delta}. Otherwise, we need a deeper analysis to compute explicitely a basis of 
$H^2(Y,\IZ)^{+}$ as explained in section \S5.

\begin{example}
 We now show that the surfaces No.\,3b and No.~5 in Table  \ref{K3chainmirror} are Dolgachev-Voisin mirror.

{\bf No.\,3b in Table  \ref{K3chainmirror}}. 
By Proposition \ref{order} and Corollary \ref{sl} a generator for $\SL(W)$  and $\widetilde{G}\cong\IZ/3\IZ$ is $\tilde{g}:=(1,14/15,7/15,3/5)$, with respect to the coordinates $x,z,y,w$. 
A local analysis in the charts shows that the point $(0:1:0:0)\in W$ is an $A_2$ singularity fixed by $\tilde{g}$, hence it induces an $A_8$ singularity in the quotient $W/\widetilde{G}$. The remaining fixed points of $\tilde g$ are $(0:0:1:0)$,     $(1:0:1:0)$ and $(-1:0:1:0)$, which give $3$ singularities of type $A_2$ in the quotient $W/\tilde G$,  two of them interchanged by $\iota$. Thus $X$ contains $12$ $j$-orbits of exceptional curves and  $r=13$. 
The automorphism $\tilde{g}$ clearly preserves the curve $C$ and 
it fixes two points on it (corresponding to the $A_8$ singularity and to the first $A_2$ singularity). By Riemann-Hurwitz formula, its image $q(C)$ has genus $3$. In conclusion the invariants of $j$ are $(r,a,\delta)=(13,3,1)$ (here $\delta$ is uniquely determined, see Figure \ref{ord2}), thus $Y$ and the surface No. 5. in Table \ref{K3chainmirror} are BCHR-mirror  
and belong to Dolgachev-Voisin   mirror families.
\par {\bf No.~5 in Table  \ref{K3chainmirror}}. 
We recall that $W$ has one $A_3$ and $3A_1$ singularities fixed by $\iota$.
By Proposition \ref{order} and Corollary \ref{sl} a generator for $\SL(W)/(J_W)$ is  $\tilde{g}:=(1,20/21,10/21,4/7)\in\SL(W)$.  
A local analysis in the charts shows that the point $(0:1:0:0)\in W$ 
is an $A_3$ singularity fixed by $\tilde g$, which gives an $A_{11}$ singularity in $W/\widetilde{G}$. 
Moreover, the points $(1:0:1:0),(0:0:0:1)$ are smooth points in $W$ fixed by $\tilde{g}$, thus giving two $A_2$ singularities of the quotient. The $3$ singularities of type $A_1$ are permuted by $\tilde{g}$ and give a point of type $A_1$ in the quotient. 
Thus $Y$ contains $11+2\cdot 2+1=16$ orbits of exceptional curves. Moreover, the genus of the curve of maximal genus is  $2$, so that the invariants of $j$ are $(r,a,\delta)=(17,1,1)$, so $Y$ and the surface No.~3b in Table  \ref{K3chainmirror} are  BCHR-mirror  and belong to Dolgachev-Voisin  mirror families. \end{example}

\begin{example} We now show that the surfaces No.\,3 and No.~18 
in Table  \ref{K3cfmirror} belong to Dolgachev-Voisin   mirror families.

{\bf No.\,3 in Table  \ref{K3cfmirror}}.  
In this case we need a deeper analysis to determine the invariant $\delta$.
The involution $\sigma$ induces the involution $\bar\sigma(y,z,w)=(y,z,-w)$ in $\IP:=\IP(3,1,1)$ 
and the involution $\iota$ induces an involution $\bar\iota$ on $W/\langle \sigma \rangle$ and $Y$.
We observe that we have the following commutative diagram. 
The map $Y\to W/\langle \sigma \rangle$ is the minimal resolution and 
$b\circ r:Z\to \IP/\langle \bar\sigma\rangle$ is obtained composing the minimal resolution 
$r$ of $\IP/\langle \bar\sigma\rangle$ with the blow-up $b$ of the singular locus of $r^*B$,
where $B$ is the branch locus of $\bar\pi$.
\begin{eqnarray}\label{proj}
\xymatrix{
W\ar[r]\ar[d]_{\pi} & W/\langle \sigma \rangle\ar[d]^{\bar \pi} & Y\ar[l]\ar[d]^{\bar \pi}\\
\IP\ar[r] & \IP/\langle\bar\sigma\rangle & Z\ar[l]
}
\end{eqnarray}
Observe that $\IP/\langle \bar\sigma\rangle\cong \IP(3,1,2)$ and $B$ is the union 
of the curves $B_1,B_2$ defined by $f(y,z,w)=0$ and $w=0$, which intersect 
at three smooth points and at the singular point $Q_1:=(1:0:0)$.
The projective plane $\IP(3,1,2)$ has a singular point of type $A_2$ at $Q_1$ and a 
point of type $A_1$ at $Q_2=(0:0:1)$, thus its resolution is a toric variety with Picard number $4$.
 Moreover $r^*B$ has $6$ double points, thus the surface $Z$ has Picard number $10$ 
 and its Picard lattice can be explicitely computed as in Example \ref{tor}.
 The invariant lattice $H^2(Y,\IZ)^+$ has rank $10$ and, by Lemma \ref{invlatt},
 it is the lattice obtained by adding to $\pi^*\Pic(Z)$ the classes of the two 
 rational curves in $\Fix(\bar\iota)$.
 An explicit computation, following the method explained in \S 5, gives that $\delta=0$.

\end{example}

We discuss one more case in detail, since here the group acting on the surface $W^T$ is not  cyclic.
\begin{example}\par {\bf No.~1 in Table \ref{K3fermatmirror}.} The equation 
$
x^2=y^6+z^6+w^6
$ 
defines a smooth K3 surface $W$ of degree $d=6$  in  $\IP(3,1,1,1)$ and $(r,a,\delta)=(1,1,1)$.
The group $\widetilde{G}=\SL(W)/J_W$ is of order $12$. By Nikulin's classification \cite{nikulin} of 
finite abelian groups acting symplectically on a K3 surface we have that 
$\widetilde{G}\cong\IZ/2\IZ\times\IZ/6\IZ$. The group $J_W$ is generated by the element $(1/2,1/6,1/6,1/6)$ and observe that the elements 
$(1/2,1/2,1,1)$ and $(1,1/6,5/6,1)$  generate  $\widetilde{G}$. Denote by $(1,0)$ the generator of order $2$ and by $(0,1)$ the generator of order $6$. Again by \cite{nikulin} we know that we have 
the following configuration of fixed points:
\begin{eqnarray*}
\xymatrix{
H^2_{(0,3)}(6)\ar@{-}[r]&H^6_{(0,1)}(2)\ar@{-}[r]&H^3_{(0,2)}(0)\ar@{-}[r]\ar@{-}[d]&H^6_{(1,2)}(2)\ar@{-}[r]&H^2_{(1,0)}(6)\\
&&H^6_{(1,1)}(2)\ar@{-}[r]&H^2_{(1,3)}(6)&\\
}
\end{eqnarray*}
where we follow the notation of \cite{nikulin} denoting by $H^m_{x}(t)$ the cyclic group of order $m$ with generator $x$, and $t$  denotes the number of fixed points having $H_x^m$ as stabilizer.

Looking at the diagram and by a local analysis one sees that it is enough to study 
the fixed points of the elements $(1,0)$, $(0,3)$ and $(1,3)$. The fixed points of 
$(1,0)$ are $(1:1:0:0)$, $(-1:1:0:0)$ and $(0:0:1:\xi^j)$, with $\xi=\exp(2\pi i/12)$ and $j=1,3,5,7,9,11$. The first two points are interchanged by $\iota$ and have in fact stabilizer of order $6$. 
The computation for the elements $(0,3)$ and $(1,3)$ is similar.
 We find that the quotient $W/\widetilde{G}$ has in total $3A_5$ and $3A_1$ singularities which gives $r=19$. Finally, by an easy computation, one sees that the curve $C$ contains $18$ points with stabilizer group of order  $2$ hence by Riemann-Hurwitz formula the curve  $C_1$ has genus $1$. In this case $\delta=1$ by Figure \ref{ord2}, thus the invariants for $Y$ are $(r,a,\delta)=(19,1,1)$. This shows that the surfaces $Y$ and $W$ belong to Dolgachev-Voisin  mirror families.\end{example}

In Table \ref{K3fermatmirror} and \ref{K3lfmirror} there are cases where $\SL(W)/J_W$
has non trivial proper subgroups $\widetilde{G}=G/J_W$.
By making similar computations of the Nikulin invariants $(r,a,\delta)$ for 
$W/\widetilde{G}$ and $W/\widetilde{G^T}$ one obtains that the corresponding 
minimal resolutions are mirror K3 surfaces, proving Theorem~\ref{main}
also in these cases. We specify however one more case, in which the method for computing $\delta$
uses a fake weighted projective plane.
\begin{definition}
A fake weighted projective space is a $\mathbb{Q}$-factorial toric variety with Picard number one.
\end{definition}
By \cite[Proposition 4.7]{con} (see also \cite[Corollary 2.3]{kas}) every fake weighted projective space is a quotient of a weighted projective space
by a finite group acting freely in codimension one.  
\begin{example}\par {\bf No.~30 in Table \ref{K3fermatmirror}.}
Let $W=x^2-y^4-z^8-w^8=0$ in $\mathbb P(4,2,1,1)$. The surface has two $A_1$ singular points at $(1:1:0:0), (-1:1:0:0)$ which are exchanged by $\iota$ and $\iota$ fixes a curve $C$ of genus $9$, so that $(r,a)=(2,2)$ and by Nikulin's table
$\delta=0$ . Moreover $\SL(W)/J_W\cong \IZ/2\IZ\times \IZ/4\IZ$ and it is generated by $(1/2,1/2,1,1), (1,1/4,3/4,1)$, which for simplicity we will call $(1,0)$ and $(0,1)$ respectively. By the results of the previous sections one can easily compute the values for $(r,a)$ for the surface $W$ and for its quotients by subgroups of $\widetilde{G}$. To compute the invariant $\delta$ of the quotient $W/\widetilde{G}$, one has a similar diagram as diagram \eqref{proj}, just replace $\langle \bar\sigma\rangle$ by the induced group on $\mathbb{P}:=\mathbb{P}(2,1,1)$ generated by  $\langle (1/2,1,1), (1/4,3/4,1) \rangle$. The quotient of $\mathbb{P}$ by this group is a fake weighted projective plane, with fan of its minimal resolution defined by $8$ rays (computation with MAGMA \cite{magma}):
$$
(-1,0),\quad (1,-2), \quad (1,2), \quad (0,-1), \quad (0,1), \quad (1,1), \quad (1,0), \quad (1,-1) 
$$
One thus proceeds as described in the previous sections to compute $\delta$. The results are resumed in Table
\ref{subgroupsfermatcase30}.
\end{example}


\newpage
\section{Tables}

 \begin{table}[h!]
$$
\begin{array}{r|l|l|c|c|c|c}
\mbox{No.} &(w_1,w_2,w_3,w_4)&f(y,z,w)&(r,a,\delta)&|\SL(W)|&|J_{W}|&\SL(W)/J_W\\
\hline
*1& (3,1,1,1)&y^6+z^6+w^6&(1,1,1)&72&6&\IZ/2\IZ\times\IZ/6\IZ
\\
*2&(5,2,2,1)&y^5+z^{5}+w^{10}&(6,4,0)&50&10&\IZ/5\IZ
\\
*8& (9,6,2,1)&y^3+z^9+w^{18}&(6,2,0)&54&18&\IZ/3\IZ\\
18&(15,10,3,2)&y^3+z^{10}+w^{15}&(10,4,0)&30&30&1
\\
26&(21,14,6,1)&y^3+z^7+w^{42}&(10,0,0)&42&42&1\\
*30&(4,2,1,1)&y^4+z^8+w^8&(2,2,0)&64&8&\IZ/2\IZ\times \IZ/4\IZ
\\
*34&(10,5,4,1)&y^4+z^5+w^{20}&(6,4,0)&40&20&\IZ/2\IZ\\
*41&(6,3,2,1)&y^4+z^6+w^{12}&(4,4,1)&48&12&\IZ/2\IZ\times\IZ/2\IZ
\\
*42&(6,4,1,1)&y^3+z^{12}+w^{12}&(2,0,0)&72&12&\IZ/6\IZ
\\
*45&(12,8,3,1)&y^3+z^8+w^{24}&(6,2,0)&48&24&\IZ/2\IZ\\
\end{array}
$$ 
\caption{The fermat mirror cases}\label{K3fermatmirror}
\end{table}

\begin{table}[h!]
$$
\begin{array}{r|l|l|c|c|c}
\mbox{No.} &(w_1,w_2,w_3,w_4)&f(y,z,w)&(r,a,\delta)&|\SL(W)|&|J_{W}|\\
\hline
*(1)1& (3,1,1,1)&y^5z+z^5w+w^5y &(1,1,1)&6&42\\
(23)3&(5,3,1,1)&y^3z+z^9w+w^7y&(3,1,1)&10&10\\
(13)11& (11,7,3,1)&y^3w+w^{19}z+z^5y&(9,1,1)&22&22\\
(11)13&(13,7,5,1)&y^3z+z^5w+w^{19}y&(11,1,1)&26&26\\
(3)23&(19,11,5,3)&y^3z+z^7w+w^9y&(17,1,1)&38&38\\
\end{array}
$$ 
\caption{The loop mirror cases}\label{K3loopmirror}
\end{table}

\begin{table}[h!]
$$
\begin{array}{r|l|l|c|c|c|c}
\mbox{No.} &(w_1,w_2,w_3,w_4)&f(y,z,w)&(r,a,\delta)&|\SL(W)|&|J_{W}|&\SL(W)/J_W\\
\hline
*1& (3,1,1,1)&y^5z+z^5y+w^6 &(1,1,1)&48&6&\IZ/8\IZ\\
*2& (5,2,2,1)&y^4z+z^4y+w^{10}&(6,4,0)&30&10&\IZ/3\IZ
\\ 
*3&(5,3,1,1)&y^3z+z^{7}y+w^{10}&(3,1,1)&40&10&\IZ/4\IZ
\\
*5&(7,4,2,1)&y^3z+z^5y+w^{14}&(7,3,1)&28&14&\IZ/2\IZ\\
6&(9,4,3,2)&y^4w+w^7y+z^6&(10,6,1)&18&18&1
\\
10&(11,6,4,1)&y^3z+z^4y+w^{22}&(10,2,1)&22&22&1
\\
*30&(4,2,1,1)&z^{7}w+w^7z+y^4&(2,2,0)&48&8&\IZ/6\IZ
\\
*31&(8,4,3,1)&z^{5}w+w^{13}z+y^{4}&(6,4,0)&32&16&\IZ/2\IZ
\\
*32&(8,5,2,1)&y^{3}w+w^{11}y+z^8&(6,2,0)&32&16&\IZ/2\IZ\\
36&(14,9,4,1)&y^{3}w+w^{19}y+z^7&(10,0,0)&28&28&1\\
*42&(6,4,1,1)&z^{11}w+w^{11}z+y^3&(2,0,0)&60&12&\IZ/5\IZ\\
47&(18,12,5,1)&z^7w+w^{31}z+y^3&(10,0,0)&36&36&1\\
\end{array}
$$ 
\caption{The loop+fermat mirror cases}\label{K3lfmirror}
\end{table}
\clearpage

\small{
\begin{table}[!t]
$$
\begin{array}{r|l|l|c|c|c}
\mbox{No.} &(w_1,w_2,w_3,w_4)&f(y,z,w)&(r,a,\delta)&|\SL(W)|&|J_{W}|\\
\hline
*(15)1& (3,1,1,1)&y^5z+z^6+w^6 &(1,1,1)&12&6\\
*(33b)2a& (5,2,2,1)&y^4z+z^5+w^{10}&(6,4,0)&20&10
\\ 
  (39)2b&(5,2,2,1)&w^8y+y^5+z^5&(6,4,0)&10&10
  \\
*(18)3&(5,3,1,1)&y^3z+z^{10}+w^{10}&(3,1,1)&20&10
\\
(35a)4&(7,3,2,2)&y^4z+z^7+w^7&(10,6,0)&14&14
\\
(24)5&(7,4,2,1)&y^3z+z^7+w^{14}&(7,3,1)&14&14\\
*(41a)6&(9,4,3,2)&y^4w+w^9+z^6&(10,6,0)&36&18
\\
*(8a)7&(9,5,3,1)&y^3z+z^6+w^{18}&(7,3,1)&36&18\\
*(7)8a&(9,6,2,1)&z^6y+y^3+w^{18}&(6,2,0)&36&18\\
(46)8b&(9,6,2,1)&w^{12}y+y^3+z^9&(6,2,0)&18&18\\
(48)8c&(9,6,2,1)&w^{16}z+z^9+y^3&(6,2,0)&18&18\\
*(1)15&(15,6,5,4)&w^6y+y^5+z^6&(12,6,1)&60&30\\
(34a)16&(15,7,6,2)&y^4w+w^{15}+z^5&(14,4,0)&30&30
\\
(19)17&(15,8,6,1)&y^3z+z^5+w^{30}&(11,1,1)&30&30\\
*(3)18&(15,10,3,2)&w^{10}y+y^3+z^{10}&(10,4,0)&60&30
\\
(17)19&(15,10,4,1)&z^5y+y^3+w^{30}&(9,1,1)&30&30\\
(5)24&(21,14,4,3)&z^7y+y^3+w^{14}&(13,3,1)&42&42\\
(45b)25&(21,14,5,2)&z^8w+y^3+w^{21}&(14,2,0)&42&42\\
(36)26a&(21,14,6,1)&w^{28}y+y^3+z^7&(10,0,0)&42&42\\
(47)26b&(21,14,6,1)&w^{36}z+y^3+z^7&(10,0,0)&42&42\\
(42b)29&(33,22,6,5)&w^{12}z+z^{11}+y^3&(18,0,0)&66&66\\
\hline
\hline
*(35b)30a&(4,2,1,1)&z^{7}w+w^8+y^4&(2,2,0)&16&8
\\
*(43)30b&(4,2,1,1)&z^{6}y+y^4+w^8&(2,2,0)&16&8
\\
*(31a)31a&(8,4,3,1)&z^{4}y+y^4+w^{16}&(6,4,0)&32&16
\\
*(34b)31b&(8,4,3,1)&z^{5}w+w^{16}+y^4&(6,4,0)&32&16
\\
*(45a)32&(8,5,2,1)&y^{3}w+w^{16}+z^8&(6,2,0)&32&16\\
*(41b)33a&(10,5,3,2)&z^{6}w+w^{10}+y^4&(8,6,1)&40&20\\
*(2a)33b&(10,5,3,2)&z^{5}y+y^4+w^{10}&(8,6,1)&40&20
\\
(16)34a&(10,5,4,1)&w^{15}y+y^4+z^{5}&(6,4,0)&20&20
\\
*(31b)34b&(10,5,4,1)&w^{16}z+z^{5}+y^4&(6,4,0)&40&20
\\
(4)35a&(14,7,4,3)&w^{7}y+y^{4}+z^7&(10,6,0)&28&28
\\
*(30a)35b&(14,7,4,3)&w^{8}z+z^7+y^4&(10,6,0)&56&28
\\
(26a)36&(14,9,4,1)&y^{3}w+w^{28}+z^7&(10,0,0)&28&28\\
(2b)39&(20,8,7,5)&z^{5}w+y^5+w^8&(14,4,0)&40&40
\\
\hline
\hline
*(6)41a&(6,3,2,1)&w^{9}y+y^4+z^6&(4,4,1)&24&12
\\
*(33a)41b&(6,3,2,1)&w^{10}z+z^6+y^4&(4,4,1)&24&12\\
*(44)42a&(6,4,1,1)&z^{8}y+y^3+w^{12}&(2,0,0)&24&12
\\
(29)42b&(6,4,1,1)&z^{11}w+w^{12}+y^3&(2,0,0)&12&12\\
*(30b)43&(12,5,4,3)&y^4z+z^6+w^8&(10,6,1)&48&24
\\
*(42a)44&(12,7,3,2)&y^3z+z^8+w^{12}&(10,4,0)&48&24
\\
*(32)45a&(12,8,3,1)&w^{16}y+y^3+z^8&(6,2,0)&48&24\\
(21)45b&(12,8,3,1)&w^{21}z+z^8+y^3&(6,2,0)&24&24\\
(8b)46&(18,11,4,3)&y^3w+w^{12}+z^9&(14,2,0)&36&36\\
(26b)47&(18,12,5,1)&z^7w+w^{36}+y^3&(10,0,0)&36&36\\
(8c)48&(24,16,5,3)&z^9w+w^{16}+y^3&(14,2,0)&48&48\\
\end{array}
$$
\caption{The chain+fermat mirror cases}\label{K3cfmirror}
\end{table}
}

\begin{table}[h!]
$$
\begin{array}{r|l|l|c|c|c}
\mbox{No.} &(w_1,w_2,w_3,w_4)&f(y,z,w)&(r,a,\delta)&|\SL(W)|&|J_{W}|\\
\hline
(27)1&(3,1,1,1)&y^5z+z^5w+w^6&(1,1,1)&6&6\\
(37)2&(5,2,2,1)&w^8z+z^4y+y^5&(6,4,0)&10&10
\\
(28)3a& (5,3,1,1)&y^3z+z^9w+w^{10}&(3,1,1)&10&10\\
*(5)3b& (5,3,1,1)&z^7y+y^3w+w^{10}&(3,1,1)&30&10
\\
*(30)4& (7,3,2,2)&y^4z+z^6w+w^7&(10,6,0)&42&14
\\
*(3b)5& (7,4,2,1)&z^7+zy^3+yw^{10}&(7,3,1)&42&14
\\
(14)7&(9,5,3,1)&w^{13}y+y^3z+z^6&(7,3,1)&18&18\\
(38)8&(9,6,2,1)&w^{16}z+z^6y+y^3&(6,2,0)&18&18\\
(17)11&(11,7,3,1)&z^5y+y^3w+w^{22}&(9,1,1)&22&22\\
(31a)12&(13,6,5,2)&z^4y+y^4w+w^{13}&(14,4,0)&26&26
\\
(19)13&(13,7,5,1)&y^3z+z^5w+w^{26}&(11,1,1)&26&26\\
(7)14&(13,8,3,2)&w^{13}+wy^3+yz^6&(13,3,1)&26&26\\
(31b)16&(15,7,6,2)&y^4w+w^{12}z+z^5&(14,4,0)&30&30
\\
(11)17&(15,8,6,1)&z^5+zy^3+yw^{22}&(11,1,1)&30&30\\
(13)19&(15,10,4,1)&y^3+yz^5+zw^{26}&(9,1,1)&30&30\\
(32)25&(21,14,5,2)&z^8w+w^{14}y+y^3&(14,2,0)&42&42\\
(1)27& (25,10,8,7)&y^5+yz^5+zw^6 &(19,1,1)&50&50\\
(3a)28& (27,18,4,5)&y^3+yz^9+zw^{10}&(17,1,1)&54&54\\
\hline
\hline
*(4)30&(4,2,1,1)&y^4+yz^6+zw^7&(2,2,0)&24&8
\\
(12)31a&(8,4,3,1)&y^4+yz^4+zw^{13}&(6,4,0)&16&16
\\
(16)31b&(8,4,3,1)&y^4+yw^{12}+wz^5&(6,4,0)&16&16
\\
(25)32&(8,5,2,1)&z^8+zw^{14}+wy^3&(6,2,0)&16&16\\
(47)36&(14,9,4,1)&y^3w+w^{24}z+z^7&(10,0,0)&28&28\\
(2)37&(16,5,7,4)&w^8+wz^4+zy^5&(14,4,0)&32&32
\\
(8)38&(16,9,5,2)&w^{16}+wz^6+zy^3&(14,2,0)&32&32\\
(42)40&(22,13,5,4)&y^3z+z^8w+w^{11}&(18,0,0)&44&44\\
\hline
\hline
(40)42&(6,4,1,1)&y^3+yz^8+zw^{11}&(2,0,0)&12&12\\
(36)47&(18,12,5,1)&y^3+yw^{24}+wz^7&(10,0,0)&36&36\\
\end{array}
$$ 
\caption{The chain mirror cases}\label{K3chainmirror}
\end{table}

\begin{table}[!ht]
$$
\begin{array}{c|c|c|c|c|c}
\widetilde{G} & \rm{generators} & (r,a,\delta) & \widetilde{G^T} & \rm{generators} & (r,a,\delta)\\
\hline
\IZ/2\IZ & (1/2,1/2,1,1) & (8,6,1) & \IZ/6\IZ & (1,1/6,5/6,1) & (12,6,1)\\\hline
\IZ/2\IZ & (1,1/2,1/2,1) & (8,6,1) & \IZ/6\IZ & (1/2,1/3,1/6,1) & (12,6,1) \\\hline
\IZ/2\IZ & (1/2,1,1/2,1) & (8,6,1) & \IZ/6\IZ & (1/2,5/6,2/3,1) & (12,6,1)\\\hline
\IZ/3\IZ & (1,1/3,2/3,1) & (7,7,1) & \IZ/2\IZ\times\IZ/2\IZ & \substack{(1/2,1/2,1,1),\\(1,1/2,1/2,1)} & (13,7,1)
\end{array}
$$
\caption{The subgroups in the fermat case No. 1}\label{subgroupsfermatcase1}
\end{table}

\begin{table}[!ht]
$$
\begin{array}{c|c|c|c|c|c}
\widetilde{G} & \rm{generators} & (r,a,\delta) & \widetilde{G^T} & \rm{generators} & (r,a,\delta)\\
\hline
0 & 0  & (2,2,0) &  \IZ/2\IZ\times \IZ/4\IZ & (1,0), (0,1) & (18,2,0)
\\
\hline
\IZ/2\IZ & (1,0) & (10,6,0) & \IZ/4\IZ & (1,1) & (10,6,0)\\\hline
\IZ/4\IZ & (0,1) & (10,6,1) &  \IZ/2\IZ & (1,2) & (10,6,1) \\\hline
\IZ/2\IZ & (0,2) & (6,6,1) & \IZ/2\IZ\times \IZ/2\IZ & (1,0), (0,2) & (14,6,1)\\
\end{array}
$$
\caption{The subgroups in the fermat case No. 30}\label{subgroupsfermatcase30}
\end{table}
\newpage
\bibliographystyle{amsplain}
\bibliography{Biblio}

\end{document}